\documentclass[12pt]{amsart}
\usepackage{amsmath}
\usepackage{tikz-cd,enumerate,ulem}
\usepackage{mathtools}
\usetikzlibrary{arrows}
\usepackage[utf8]{inputenc}
\usepackage{xspace}
\usepackage{comment}

\usepackage[english]{babel}
\usepackage[T1]{fontenc}
\usepackage{amsmath,amssymb}

\usepackage[a4paper,top=3cm,bottom=2cm,left=3cm,right=3cm,marginparwidth=1.75cm]{geometry}

\usepackage{amsthm}
\usepackage{amsfonts,dirtytalk,marginnote}
\usepackage{graphicx}
\usepackage[colorinlistoftodos]{todonotes}
\usepackage[hidelinks]{hyperref}
\usepackage{thmtools}

\usepackage{crossreftools}

\usepackage{url}

\newtheorem{theorem}{Theorem}[section]
\newtheorem{thm}[theorem]{Theorem}

\newtheorem{proposition}[theorem]{Proposition}
\newtheorem{definition}[theorem]{Definition}
\newtheorem{lemma}[theorem]{Lemma}

\newtheorem{corollary}[theorem]{Corollary}

\theoremstyle{definition}
\newtheorem{ex}[theorem]{Example}

\newtheorem{remark}[theorem]{Remark}

\theoremstyle{plain}
\newtheorem{maintheorem}{Theorem}

\def\AA{\mathbb{A}}
\def\CC{\mathbb{C}}

\def\FF{\mathbb{F}}
\def\GG{\mathbb{G}}

\def\KK{\mathbb{K}}
\def\LL{\mathbb{L}}
\def\MM{\mathbb{M}}
\def\NN{\mathbb{N}}
\def\QQ{\mathbb{Q}}

\def\ZZ{\mathbb{Z}}

\def\calo{\mathcal{O}}

\def\Gal{\mathrm{Gal}}
\def\GL{\mathrm{GL}}

\def\M{\mathrm{M}}

\def\SL{\mathrm{SL}}

\def\Qbar{\overline{\QQ}}

\newcommand{\mug}{{\boldsymbol\mu}}

\usepackage{enumitem}
\setlist[itemize,1]{label={--\,}}
\setlist{leftmargin=8mm}

\newtheorem*{fact*}{Fact} 

\newcommand{\Z}{{\mathbb{Z}}}

\newcommand{\Q}{{\mathbb{Q}}}
\newcommand{\Qp}{\overline\Q_p}

\newcommand\B{{\rm (B)}\xspace}
\newcommand\SB{{\rm (SB)}\xspace}
\newcommand\ADZ{{\rm (ADZ)}\xspace}

\newcommand\id{{\mathrm{id}}}

\newcommand{\vareps}{\varepsilon}

\newcommand{\into}{\hookrightarrow}

\newcommand{\ovl}{\overline}
\newcommand{\wtl}{\widetilde}

\newcommand{\ccirc}{\kern0.5ex\vcenter{\hbox{$\scriptstyle\circ$}}\kern0.5ex}

\newtheorem*{theorem*}{Theorem}
\newtheorem*{remark*}{Remark}

\numberwithin{equation}{section}

\author{Andrea Conti, Pietro Piras, and Lea Terracini}
\title{The Bogomolov property for $p$-supercuspidal eigenforms}
\keywords{Weil height, Bogomolov property, supercuspidal Hecke eigenforms}
\subjclass[2020]{11G50, 11F80}

\begin{document}

\begin{abstract}
We prove a lower bound on the Weil height, the so-called Bogomolov property, for the algebraic extensions of $\QQ$ cut out by the adelic Galois representations attached to certain eigenforms whose local component at a prime $p$ is supercuspidal. To this end, we give a method for constructing metric inequalities over $p$-adic Lie extensions of fields over $\QQ$ that are finitely ramified at $p$.
\end{abstract}

\maketitle

\setcounter{tocdepth}{1}
\tableofcontents

\section*{Introduction}
A classical problem in arithmetic geometry and diophantine arithmetic is to understand lower bounds for
heights. On the arithmetic side, the \emph{Weil height} is a measure of the complexity of an algebraic number. Denoting by $h(\alpha)$ the Weil height of an algebraic number $\alpha$, an important conjecture of Lehmer predicts the existence of a positive constant $C$ such that
$$h(\alpha)\geq\frac{C}{\deg(\alpha)}$$
for all non-zero algebraic numbers $\alpha$ that are not roots of unity (see Smyth's excellent review \cite{S08} for a survey on Lehmer's conjecture). Lehmer's conjecture, and its generalizations due to Rémond \cite[Conjecture 3.4]{Rem17}, have been the object of various partial results, as seen, for instance, in \cite{Rem17,Amo16,Gri17,Pot21,Ple24,Ple22,Ple24a}. A natural analogue on the geometric side is to consider an elliptic curve $E$ defined over a number field, equipped with the Néron-Tate height $\hat{h}$. Lang conjectured a lower bound on $\hat{h}$ depending on the discriminant of $E$; Szpiro conjectured another bound on the norm of the discriminant of $E$, implying Lang's bound and among other things the ABC conjecture; see \cite{S10} for a survey. 

In higher dimension, the Bogomolov conjecture proposes a description of subvarieties of abelian varieties containing many points of small canonical height. This inspired Bombieri and Zannier \cite{BombieriZannier2001} to propose an arithmetic counterpart, for the Weil height over an algebraic extension $\LL$ of $\QQ$ (or in other words, with the $\LL$-points of the abelian variety replaced with $\GG_m(\LL)$). Namely, a field $\LL$ of algebraic numbers is said to have the \textit{Bogomolov property} (property \B for short) if there exists a positive constant
\(c(\LL)>0\) such that every non-zero element \(\alpha\in \LL^\times\), which is
not a root of unity, satisfies
\[
h(\alpha)\ge c(\LL).
\]
It is a hard problem to determine whether an arbitrary $\LL$ has property \B or not. 
By Northcott's theorem, every number field has (B), so the interesting examples are all of infinite degree. Note that fields with (B) are fields in which Lehmer's conjecture is trivially true, whereas Lehmer's conjecture makes a prediction on the whole $\overline{\Q}$, that does not have (B).

A series of works by many people produced various classes of infinite algebraic extensions $\LL/\QQ$ with the Bogomolov property (see the introductions of \cite{Piras2026} and \cite{Pl24}). It is particularly interesting to investigate this property in the case when $\LL$ originates from a geometric object, for instance, when it is generated by the coordinates of the torsion points of an abelian variety over a number field. In such a setting, the field $\LL$ can be written as the fixed field of the kernel of a representation of an absolute Galois group.

While property \B for arbitrary abelian varieties is the object of an open conjecture of David (see \cite[p.114]{Frey2021}), it is known for the field generated by the torsion points of an
elliptic curve $E$ over $\QQ$ thanks to the work of Habegger \cite{Habegger2013}.

An analogous result for the field
\[
\mathbb Q(\rho_f)\coloneqq \ovl\Q^{\ker\rho_f},
\]
fixed by the kernel of the adelic Galois representation attached to a modular form $f$, was proved by Amoroso and Terracini \cite{AmorosoTerracini2024} in the case of trivial nebentypus, and by Piras \cite{Piras2026} in the general case. Both results require the additional assumption that $f$ admits a \emph{strong supersingular prime}, that is, a prime $p$ not dividing the level and such that the $p$-th Hecke eigenvalue $a_p(f)$ is 0.

Here and below, we only consider cuspidal non-CM eigenforms $f$: when $\rho_f$ is either abelian or induced, property \B for $\mathbb Q(\rho_f)$ is already known by \cite{AmorosoDvornicich2000} (in the abelian case) and \cite{AmorosoZannier2000} (in the induced case): indeed, $\mathbb Q(\rho_f)$ is then abelian over a quadratic extension $\KK$ of $\QQ$.

The hypothesis that $a_p(f)=0$ for some prime not dividing the level is rather restrictive. Indeed, for weight strictly larger than \(2\)
it occurs only rarely: for a fixed $p$, Calegari and Sardari \cite{CalegariSardari} show that there are only finitely many forms $f$ of a given level for which $a_p(f)=0$, while for a fixed $f$, the Atkin--Serre conjecture \cite[(4.11$_k$?)]{SerreDiv1976} predicts that only finitely many primes $p$ satisfy $a_p(f)=0$. 

Nevertheless, in \cite{AmorosoTerracini2024, Piras2026}  this condition plays two crucial roles. The first one is local at \(p\): it allows one to relate the local \(p\)-adic Galois representation to Lubin--Tate theory, hence to obtain a precise control of the ramification groups. This ingredient is essential in the construction of the metric inequalities which eventually force a uniform positive lower bound for the height. The second one concerns the behaviour away from \(p\): the same condition is used to prove that the image of the prime-to-$p$ representation modulo its center is  finite when restricted to a decomposition group at $p$. This allows one to apply a lemma of Amoroso, David, and Zannier \cite[Theorem 1.5]{AmorosoDavidZannier2014} and get a lower bound for the heights in the prime-to-$p$ part of the representation.

The local aspect at \(p\) has been further developed in \cite{ContiTerracini2025}. There, the original argument of Habegger based on Lubin--Tate theory is extended to the setting of a \(p\)-adic Lie extension $L$ of a $p$-adic field $K$, under the assumption that $L/K$ is totally ramified at a place above \(p\). This allowed the authors to prove property \B\ for a wide class of $p$-adic Galois representations having a big local image at $p$.

However, it is much less clear how to extend the argument for the prime-to-$p$ part to modular Galois representations beyond the strongly supersingular case.

The starting point of the present paper is the observation that there are indeed certain modular forms for which the representation away from \(p\) has a finite local projective image at $p$. These are the modular forms for which the associated admissible representation of $\GL_2(\Q_p)$, provided by the local Langlands correspondence, is \emph{supercuspidal}. This provides a natural replacement for the condition of strong supersingularity at $p$; although the local representation at \(p\) is no longer described by  Lubin--Tate theory, the projective image away from \(p\) retains precisely the local finiteness at $p$ needed in the global part of the argument.

Our main result is therefore an analogue of the main theorem of \cite{AmorosoTerracini2024} in this setting, and it identifies a new family of infinite extensions, arising from modular forms, for which the Bogomolov property holds. 

More precisely, we prove the following theorem. Given an eigenform $f$, we write $\rho_{f,p}:\Gal(\ovl\QQ/\QQ)\to\GL_2(\ovl\QQ_p)$ for the associated $p$-adic Galois representation and $\pi_{f,p}$ for the associated admissible representation of $\GL_2(\QQ_p)$. We also write $G_p\subset\Gal(\ovl\QQ/\QQ)$ for a decomposition group at $p$.

\begin{maintheorem}[{cf. Theorem \ref{thm:Bsupercusp}}]
Let $k\ge 2,N\ge 1$ be two integers, and  let $f$ be a non-CM, cuspidal eigenform of weight $k$ and level $\Gamma_0(N)$ such that:

\begin{enumerate}
    \item all of the Hecke eigenvalues of $f$ belong to $\QQ$;
    \item there exists an odd prime $p$ such that the admissible representation $\pi_{f,p}$ is supercuspidal, and if $\rho_{f,p}\vert_{G_{p}}$ is induced, then  $p\geq 5$ and the image of the mod $p$ reduction $\ovl\rho_{f,p}$ contains $\SL_2(\FF_p)$.
\end{enumerate}
Then the corresponding adelic representation
\[\rho_f:G_\QQ \longrightarrow \GL_2(\widehat\ZZ)\]
has property \B.
\end{maintheorem}

Actually, the conclusion of Theorem \ref{thm:Bsupercusp} is stronger: property \B holds for every finite extension of $\QQ(\rho_f)$.

In order to combine the local and global ingredients, we are led to generalize the analysis of \(p\)-adic Lie extensions from \cite{ContiTerracini2025} to a setting in which the base field is no longer a finite extension of \(\mathbb Q_p\), but rather only finitely ramified. This is necessary since the level of a modular form which is supercuspidal at $p$ is divisible by $p$, so the field cut out by the prime-to-$p$ part of the representation is in general (finitely) ramified at $p$. We show that, in this context, Sen's theorem still provides the required asymptotic control of the ramification filtration. More precisely, at each level of the Lie filtration, the maximal index of the ramification groups dominates the ramification index.  This allows us to construct a system of metric inequalities in the spirit of Habegger's method. Their combined effect is to force a uniform positive lower bound for the height of non-torsion elements in the Lie extension generated by the $p$-adic representation. We refer to Theorems \ref{thm:Btower} and \ref{cor:composto2} for our precise statements. We believe that these results are of independent interest and may find applications in establishing height lower bounds in a variety of settings beyond that of Theorem A. For instance, they are used in \cite{ContiDelCorsoPlessisTerracini2026} to describe points of small height in fields obtained by adjoining radicals to certain algebraic extensions of $\mathbb{Q}$.

A main limitation in our work is the assumption that  the Hecke field of the supercuspidal form is $\QQ$; removing this assumption would need a generalization of our argument in many points, and will be the object of future work. 

We also point out that the assumptions in Theorem A are more restrictive when the local Galois representation associated to $f$ is induced. However, such representations may be rare, in view of the following consideration. Consider an irreducible modular Galois representation
\(
\rho_f:G_\QQ \to \GL_2(\Qbar_p)
\)
of weight \(k\geq 2\).  A well-known open question, attributed to Ralph Greenberg \cite[(1.2)]{GhateVatsal2004}, asks whether
$$\rho_f|_{G_p} \hbox{ splits } \Longrightarrow f \hbox{ has CM.} $$ This question was revisited by Calegari and Sardari in \cite[\S 1.2]{CalegariSardari}, who wonder if the condition of \emph{splitting} can be replaced by that of \emph{being induced}, up to a sparse set of exceptions (including  representations arising from supersingular modular forms of weight 2). Although their main result only concerns the crystalline case, with $a_p=0$, their question is significant in a more general situation. If a similar result can be proved in the potentially crystalline case, then Theorem \ref{thm:Bsupercusp} would imply that \emph{almost all} modular forms that are supercuspidal at $p$ and have rational Hecke field  give rise to Galois representations satisfying property \B.

\subsection*{Overview of the structure.} The paper is organized as follows. Section \ref{sec:prelim} introduces the main definitions and recalls some relevant facts about property \B. In Section \ref{sec:sen}, we study totally ramified \(p\)-adic Lie extensions of finitely ramified extensions of \(\mathbb Q_p\), proving a regularity property for their ramification groups. Sections \ref{sect:towers} and \ref{sect:compositum} establish criteria for property \(\mathrm{(B)}\) for Galois algebraic extensions whose localizations at \(p\) belong to the class described above. In Section \ref{sec:outsidep} we analyze the local Galois representations attached to modular forms which are supercuspidal at \(p\), with particular emphasis on  the finiteness of the projective image away from \(p\). In Section \ref{sec:padic}, we investigate the $p$-adic Galois representation attached to a $p$-supercuspidal eigenform. Finally, in Section \ref{sec:B}, we combine the local $p$-adic properties with the prime-to-$p$ study and prove the Bogomolov property for the field cut out by the adelic representation. Section \ref{sec:examples} contains some examples of eigenforms that satisfy all of the requirements of our main theorem.

\subsection*{Notation.} In all that follows, $p$ is a prime. Note that $p$ is assumed to be odd in Sections \ref{sec:outsidep} and \ref{sec:padic}. \\
Algebraic extensions of $\QQ_p$ will be denoted by $F,K,E,L,\ldots$ while algebraic extensions of $\QQ$ will be denoted by $\mathbb{F},\mathbb{K},\mathbb{E},\mathbb{L},\ldots$ \\
If $\alpha$ is an algebraic number, we denote by $h(\alpha)$ the absolute logarithmic Weil height of $\alpha$. We recall its definition in Section \ref{sec:prelim}. \\
For an arbitrary field $\kappa$, we denote by $\ovl \kappa$ an algebraic closure of $\kappa$, and by $G_\kappa=\Gal(\ovl \kappa/\kappa)$ the absolute Galois group of $\kappa$. \\
Given a continuous representation $\rho$ of $G_\QQ$, we write $\QQ(\rho)$ for the subfield of $\ovl\QQ$ fixed by $\ker\rho$. \\
We denote by $G_p$ an absolute Galois group of $\Q_p$, and by $I_{p}\subset G_{p}$ its inertia subgroup. We write $I_p^w$ for the wild inertia subgroup of $I_p$, and $I_p^t$ for the tame quotient $I_p/I_p^w$. We fix once and for all an embedding $\ovl\Q\into\Qp$, identifying $G_{p}$ with a decomposition subgroup of $G_\Q$. We use the same embedding to identify $G_{\KK_v}$ with a decomposition group of $G_\KK$ for any number field $\KK$, where $v$ is the place
determined by the fixed embedding. \\
We denote by $\mug$ the set of roots of unity in $\ovl\Q$ and by $\mug_{p^\infty}$ the subgroup of roots of unity of order a power of $p$. For every field $\kappa$, we put  $\mug_{p^\infty}(\kappa)=\kappa^\times\cap\mug_{p^\infty}$. We write $\vareps_p\colon G_\Q\to\Z_p^\times$ for the $p$-adic cyclotomic character, and use the convention that $\vareps_p$ has Hodge--Tate weight 1. \\ We denote by $Z(G)$ the center of a group $G$. \\
Given a group $G$ and a decreasing filtration $\mathcal F=(G_i)_i$ of $G$, indexed by a totally ordered set $(I,<)$, a \textit{jump} in $\mathcal F$ is an element $i\in I$ such that $G_i\supsetneq G_j$ for every $j>i$.

\subsection*{Acknowledgments.} 
We thank the referee for the careful reading of our manuscript and for the helpful comments, which have enabled us to improve the clarity and precision of the paper.\\
A.C. is funded by the Deutsche Forschungsgemeinschaft - Project-ID 444845124 – TRR 326 GAUS.\\
P.P. and L.T. are members of the Italian INdAM group GNSAGA.

\section{Preliminaries}\label{sec:prelim}

We recall the definition of the (absolute, logarithmic) Weil height $h(\alpha)$ of an algebraic number $\alpha$. Given a number field $\KK$ containing $\alpha$,
\[
h(\alpha)=\frac{1}{[\KK:\QQ]}
\sum_v [\KK_v:\QQ_{v_0}]
\log\max\{1,|\alpha|_v\},
\]
where $v$ runs over all places of $\KK$, $v_0$ is its restriction to $\QQ$, and the absolute values extend the usual absolute values on $\QQ$. It is easily checked that such a definition depends only on $\alpha$, not on $\KK$. 

Following \cite{BombieriZannier2001}, we say that:

\begin{definition}
A set $\mathcal{Y}$ of algebraic numbers has the \textup{Bogomolov property} (property \B for short) if there is a positive constant $C$ such that $h(\alpha)\geq C$ for every $\alpha\in\mathcal{Y}\setminus(\mug\cup\{0\})$. \\

Given a subfield $\LL\subset\ovl\QQ$, we say that $\LL$ has  the \textup{stable Bogomolov property} (property \SB for short) if every finite extension of $\LL$ has the Bogomolov property.
\end{definition}

In the direction of classifying fields with property \B, we recall a result from \cite{ContiTerracini2025}. Let $K,L$ be the closures in $\ovl\Q_p$ of $\KK,\LL$, respectively, via our fixed embedding $\ovl\Q\into\ovl\Q_p$, and let $v$ be the corresponding place of $\KK$.

\begin{theorem}[{\cite[Theorem 2.9, Corollary 2.11]{ContiTerracini2025}}]\label{teo:BforLie} 
Let $\KK$ be a number field, $\GG$ a compact $p$-adic Lie group, $\LL/\KK$ a Galois extension with Galois group $\GG$ and $v$ a $p$-adic place of $\LL$. Let $(\GG_n)_{n\ge 0}$ be a Lie filtration of $\GG$, and for every $n$ let $G_n$ be the intersection of $\GG_n$ with the decomposition group of $\GG$ attached to $v$. 
Assume that $\LL/\KK$:
\begin{itemize}
\item[a)] has the \emph{central element property}, i.e. there exists $\tau \in Z(\GG)$,  such that, for every $\zeta\in\mug_{p^\infty}(\LL)$, we have $\tau(\zeta)=\zeta^g$ for some $g\in\ZZ$, $g>1$;
\item[b)] is \emph{locally potentially totally ramified} at $v$, i.e. a decomposition group $D_v$ at $v$ contains the inertia subgroup $I_v$ as an open subgroup;
\item[c)] has the \emph{normal closure property}, i.e. there exists $n_0\geq 1$ such that, for every $n\geq n_0$, the normal closure of $G_{n-1}/G_{n}$ in $\GG/\GG_{n}$ is $\GG_{n-1}/\GG_{n}$
\end{itemize}
Then $\LL$ has \B. \\
In particular, if $\LL/\KK$ satisfies (a) and in addition it is \emph{potentially totally ramified} at $v$, i.e. an inertia subgroup $I_v$ at $v$ is open in $\Gal(\LL/\KK)$, then $\LL$ has \B.
\end{theorem}

Notice that when $\LL/\KK$ is potentially totally ramified at $v$, then $G_n=\GG_n$ for $n\gg 0$, so that the normal closure property is automatic, and $\LL/\KK$ is locally potentially totally ramified since $G_n$ is contained in an inertia subgroup for $n\gg 0$.

\noindent We recall a further condition appearing in the work of Amoroso--David--Zannier \cite{AmorosoDavidZannier2014}.

\begin{definition}\label{def:propADZ} Let $\KK$ be a number field, $v$ a non Archimedean place of $\KK$ and $\LL/\KK$ a Galois extension with Galois group $\GG$. We say that $\LL/\KK$ has \emph{property \ADZ} at $v$ if the extension $\LL^{Z(\GG)}/\KK$ has bounded local degrees over $v$. We say that $\LL/\KK$ has \emph{property \ADZ} if it has property \ADZ at some non Archimedean place $v$ of $\KK$.
\end{definition}

If $\LL/\KK$ has \ADZ at a place $v$ then it has \B by \cite[Theorem 1.5]{AmorosoDavidZannier2014}. Moreover, since every finite extension of $\LL$ is the compositum of $\LL$ with a finite extension of $\KK$, the following is an immediate corollary of a result of the second author.

\begin{proposition}[{\cite[Theorem 5.8]{Piras2026}}]\label{prop:ADZB}
If $\LL/\KK$ has property \ADZ at a place $v$, then it has property \SB.
\end{proposition}

\section{An application of a theorem of Sen}\label{sec:sen}

In this section, we generalize \cite[Proposition 1.3]{ContiTerracini2025} to the case when the base field $F/\QQ_p$ is not assumed to be finite, but only finitely ramified.

We first recall the standard notation for upper and lower ramification groups, following Serre \cite[Chapter IV, \S3]{Serre1968}.  Consider a finite Galois extension $L/F$ in $\Qbar_p$, such that $F$ is discretely valued. For each integer $i\geq -1$ let $G[i]$ denote the $i$-th ramification subgroup of $G\coloneqq \Gal(L/F)$: then $G[-1]=G$, and for $i\geq  0$  
\[G[i]=\{\sigma\in G\ |\ \sigma(x)\equiv x\mod \mathfrak{P}^{i+1}, \hbox{ for every } x\in\calo_L\},\] 
where $\mathfrak{P}$ is the maximal ideal in $\calo_L$. For $u\in [-1,\infty)$, we set $G[u]= G[\lceil u\rceil ] $, where $\lceil u \rceil$ is the smallest integer $\geq u$. As in \cite[Chapter IV, \S 3]{Serre1968}, we define the function $\varphi_{L/F}\colon[0,+\infty)\to[0,+\infty)$ and its inverse $\psi_{L/F}$. They  allow to define the \textit{upper numbering} of the ramification groups: \[G(u)=G[\psi_{L/F}(u)], \quad\hbox{ that is } G[u]=G(\varphi_{L/F}(u)).\]

Now let $\QQ_p\subseteq F \subseteq L\subseteq \Qbar_p$ be a tower of fields such that $L/F$ is Galois and totally ramified and  $F/\QQ_p$ has a  finite ramification index $e = e(F/\QQ_p)$. Assume that the Galois group $G = \Gal(L/F)$ is a $p$-adic Lie group, with $\dim(G)> 0$.
Let 
\begin{equation*} \ldots\subseteq G_n\subseteq G_{n-1}\subseteq \ldots\subseteq G=\Gal(L/F) \end{equation*}
be a Lie filtration on $G$, as defined in \cite[\S 3]{Sen1973}. One would like to apply Sen's theorem \cite{Sen1972}, but the original formulation only works if $F$ is a  \emph{finite} extension. We instead need the following improvement of Probst.
\begin{theorem}[Theorem 7.1, \cite{Probst2008}]\label{thm:probst}
    Let \(F\) be a complete discretely valued field of characteristic zero with perfect residue field of characteristic \(p>0\), and let \(L/F\) be a totally ramified Galois extension. Let $e=e_F=v_F(p)$ denote the absolute ramification index of $F$. Assume that the Galois group $G=\Gal(L/F)$ is $p$-adic analytic with positive dimension. Let $(G_n)$ be a filtration on $G$ that is uniformly equivalent in scaling 1 to some Lie filtration. Then there exists a constant $c>0$ such that
    \begin{equation*}
      G(ne+c)\subset G_n\subset G(ne-c)
    \end{equation*}
    for all $n$, with $G(r)=G$ for $r<0$.
\end{theorem}
Observe that $F$ is a discretely valued field with perfect residue field. Up to passing to the completion of the base field, which does not change the ramification filtration and the Galois group, we can also assume that $F$ is complete and apply Theorem \ref{thm:probst}. We conclude that there exists a constant $c > 0$ such that
\begin{equation}\label{eq:sen} G(ne + c) \subseteq  G_n\subseteq  G(ne - c),\end{equation}
for all $n$, with $G(r) = G$ for $r < 0$.

We set:
\begin{itemize}
\item $Q_n=G/G_n$ and $|Q_n|=e_n$,
\item $L_n=L^{G_n}$, so that $\Gal(L_n/F)=Q_n$,
\item $Q_n[i]=i$-th lower ramification group of $L_n/F$,
\item $Q_n(i)=i$-th upper ramification group of $L_n/F$.
\end{itemize}

Let  $r^n_1<\ldots <r^n_{s(n)}$ be the jumps in the chain $(Q_n[i])_i$, so that 
\begin{equation*}Q_n=Q_n[r_1^n]\supsetneq Q_n[r_2^n]\supsetneq \ldots\supsetneq Q_n[r_{s(n)}^n]\supsetneq \{1\}.\end{equation*}

Let $A$ be a positive integer such that for every $n$, $[G_{n-1}:G_n]\leq p^A$.
\begin{proposition}
    The number of jumps in the chain $(G(i))_i$ lying in an interval of length $\delta$ in $[0,+\infty)$ is bounded from above by $1+A\left( \frac {\delta+2(c+e)} e \right )$. In particular, the set of jumps of $L/F$ is discrete in $[0,+\infty)$.
\end{proposition}

\begin{proof}
    Let $a,b\in[0,+\infty)$ with $b-a=\delta>0$. We put $G_{-1}=G$. For every $h,k\in \NN\cup\{-1\}$ such that  $h\geq \frac{b+c}e$ and $k\leq \frac{a-c} e$ we have by \eqref{eq:sen}, 
    \begin{equation*}
        G_h\subseteq G(he-c)\subseteq G(b)\subseteq G(a)\subseteq G(ke+c)\subseteq G_k.
    \end{equation*}
    Since $[G_k:G_h]\leq p^{A(h-k)}$ there cannot be more than $A(h-k)$ positive jumps in the interval $[a,b)$ since for every positive jump the index of the ramification groups increases by a factor of at least $p$. The best possible bound will be achieved by taking the smallest $h$ bigger than $\frac{b+c}e$ and the biggest $k$ smaller than $\frac{a-c} e$: 
setting
\(
h=\left\lceil\frac{b+c}{e}\right\rceil\),
\(
k=\max\left\{-1,\left\lfloor\frac{a-c}{e}\right\rfloor\right\}
\),
we have 
\[
h-k\leq
\left(\frac{b+c}{e}+1\right)
-\left(\frac{a-c}{e}-1\right)
=\frac{\delta+2c}{e}+2.
\]
Including the possible jump at $0$, we obtain at most
\[
1+A\left(\frac{\delta+2(c+e)}{e}\right)
\]
jumps in $[a,b)$, which proves the claim.
\end{proof}

\begin{lemma}\label{lem:effe1}
    There exists $f\in\NN$ such that, for every  jump $u$ in the chain $(G(i))_i$,  there exists $\epsilon >0$ such that $[G(u):G(u+\epsilon)]\leq p^f$. 
\end{lemma}
\begin{proof}
    Let $u$ be a jump in the filtration. We want to squeeze $G(u)$ between two groups $G_h$ and $G_k$ using \eqref{eq:sen}. The best possible values for $h$ and $k$ such that
    \begin{equation*}
        G_h\subseteq G(he-c)\subsetneq G(u)\subseteq G(ke+c)\subseteq G_k
    \end{equation*}
    are 
    \begin{equation*}
        h=\left\lfloor \frac{u+c}e\right\rfloor +1,\quad\quad k=\max\left\{-1,\left \lfloor \frac{u-c} e\right \rfloor\right\}.
    \end{equation*}
    For $0<\epsilon< he-c-u$ we have $G(he-c)\subseteq G(u+\epsilon)\subseteq G(u)$ so
    \[[G(u):G(u+\epsilon)]\leq [G_k:G_h]\leq p^{A(h-k)},\]
    and \[ h-k\leq \frac {u+c} e+1 -\frac {u-c}{e}+1= \frac {2c} e +2.
    \]
    \\
    We conclude by putting $f=\lceil A\left (\frac {2c} e +2\right )\rceil$.
\end{proof}

\begin{lemma}\label{lem:effe}
There exists a positive integer $f$ such that, for every sufficiently large $n\in\NN$,
\begin{equation*}
    \begin{split}
        & p\leq |Q_n[r^n_{s(n)}]| \leq  p^f,\\
        & p\leq [Q_n[r^n_{i}]:Q_n[r^n_{i+1}]]\leq p^f   \quad \hbox{for } i=2,\ldots, s(n)-1, \\
        & 1<[Q_n:Q_n[r^n_{2}]]\leq  p^f. 
    \end{split}
\end{equation*}
\end{lemma}
\begin{proof}
    Choose $f\in\NN$ with the property given in Lemma \ref{lem:effe1}. Let $u$ be a jump in the filtration, and let $\epsilon>0$ be as in Lemma \ref{lem:effe1}. Since each $Q_n(u)$ is a quotient of $G(u)$ we have
    \begin{equation*}
        \frac{Q_n(u)}{Q_n(u+\epsilon)}=\frac{G(u)G_n}{G_n}\bigg /\frac{G(u+\epsilon)G_n}{G_n}\simeq \frac{G_nG(u)}{G_nG(u+\epsilon)}.
    \end{equation*}
    Using Dedekind's modular law we have
    \begin{equation*}
        G(u)\cap G(u+\epsilon)G_n= G(u+\epsilon)(G(u)\cap G_n)
    \end{equation*}
    that shows
    \begin{equation*}
        \frac{G(u)}{G(u+\epsilon)}\twoheadrightarrow\frac{G(u)}{G(u+\epsilon)(G(u)\cap G_n)}\simeq \frac{G_nG(u)}{G_nG(u+\epsilon)}.
    \end{equation*}
    Since
    \begin{equation*}
        [Q_n(u):Q_n(u+\epsilon)]=\left[G_nG(u):G_nG(u+\epsilon)\right]\leq [G(u):G(u+\epsilon)],
    \end{equation*}
    for any pair of adjacent
    jumps $r<s$ in the chain $(Q_n[i])_i$ we have by Lemma \ref{lem:effe1},
    \[
        1<[Q_n[r]:Q_n[s]]\leq p^f,
    \]
    and the lower bound can be replaced by $p$ if $r>0$. The same argument bounds the order of the last nontrivial ramification group, proving the claim.
\end{proof}

\begin{remark} Assume that $F/\QQ_p$ is finite and let $k_F$ be the residue field of $F$. By \cite[Chapter IV, \S 2]{Serre1968}, the quotient $Q_n/Q_n[1]$ is isomorphic to a subgroup of $k_F^\times$, and for $j\geq 1$, the quotients $Q_n[j]/Q_n[j+1]$ are isomorphic to subgroups of $k_F$. Therefore, one can take $f=[k_F:\FF_p]$ in Lemma \ref{lem:effe}. This is precisely the choice made in the proof of \cite[Theorem 1.1]{ContiTerracini2025}.
\end{remark} 

\begin{proposition}\label{teo:bound1}

Keep the above notation. For $j_0\in\NN$, there exists a constant $C>0$ (depending on $j_0$) such that for every sufficiently large $n$,
\begin{equation*}  \frac{r^n_{{s(n)}-j_0}} {e_n}\geq C.\end{equation*}
\end{proposition}

\begin{proof}
The proof is the same as in \cite[Theorem 1.1]{ContiTerracini2025}, using Lemma \ref{lem:effe} in place of the finiteness assumption of $F/\QQ_p$. We report it here for completeness.

Choose a positive integer $f$ as in Lemma \ref{lem:effe}.
For every sufficiently large $n$, we have that
\begin{equation}\begin{split}\label{eq:pf}
& p\leq |Q_n[r^n_{s(n)}]| \leq  p^f,\\
& p\leq [Q_n[r^n_{i}]:Q_n[r^n_{i+1}]]\leq p^f   \quad \hbox{
for } i=2,\ldots, s(n)-1,
\\
& 1<[Q_n:Q_n[r^n_{2}]]\leq  p^f. 
\end{split} \end{equation}
Write $h_i^n=|Q_n[i]|$, $\varphi_n=\varphi_{L_n/F}$,
and $\psi_n=\psi_{L_n/F}$.
Then
\begin{align}\label{hn} h^n_{r^n_{s(n)-j}}\leq  p^{f(j+1)} & \hbox{ for } j=0,\ldots, s(n)-1.
\end{align}

Let $0\leq j_0\leq s(n)-2$.
Then for $j=1,\ldots ,s(n)-j_0-1$
\begin{equation}\label{eq:dis0} \begin{split} \varphi_n(r^n_{{s(n)}-j_0})& -\varphi_n(r^n_{s(n)-j_0-j}) =\frac 1 {e_n}(h^n_1+\ldots + h^n_{r^n_{{s(n)}-j_0}})-\frac 1 {e_n}(h^n_1+\ldots +h^n_{r^n_{s(n)-j_0-j}})\\
& = \frac 1 {e_n}(h^n_{r^n_{s(n)-j_0-j}+1}+\ldots +h^n_{r^n_{s(n)-j_0}})\\
&= \frac 1 {e_n}(r^n_{s(n)-j_0-j+1}-r^n_{s(n)-j_0-j})h^n_{r^n_{s(n)-j_0-j+1}}+\ldots + \frac 1 {e_n}(r^n_{s(n)-j_0}-r^n_{s(n)-j_0-1})h^n_{r^n_{s(n)-j_0}},\intertext{where we used \cite[Ch. IV, \S3]{Serre1968} in the first line, and then the fact that $h^n_{a}=h^n_{b}$ for $r^n_i<a,b\le r^n_{i+1}$, by definition of the jumps $r^n_i$, }
& \leq \frac 1 {e_n}(r^n_{s(n)-j_0}-r^n_{s(n)-j_0-j}) h^n_{r^n_{s(n)-j_0-j+1}}
\intertext{because $h^n_{a}\ge h^n_b$ when $a\le b$,}
& \leq \frac {r^n_{s(n)-j_0}} {e_n} p^{f(j_0+j)},\quad \hbox{by } \eqref{hn}.
\end{split}\end{equation}

Put $d=\dim G>0$. Since $(G_n)$ is a Lie filtration, for every fixed $k_0$ and all sufficiently large $n$ we have
\[
[G_{n-k_0}:G_n]=p^{dk_0}.
\]
Choose $k_0\geq 1$ such that $dk_0\geq f(j_0+1)$. Consider the chain
\begin{equation*} G_n\subsetneq G_{n-k_0}\subseteq G(e(n-k_0)-c)
\subseteq G(e(n-k_0-1)-c)\subseteq G_{n-k_0-1-\lceil \frac {2c}e\rceil}.
\end{equation*}
Passing to the quotient by $G_n$ we get
\[\{1\}\subsetneq Q_n (e(n-k_0)-c)=Q_n[\psi_n(e(n-k_0)-c)].\]
Since
\[
|Q_n(e(n-k_0)-c)|\geq [G_{n-k_0}:G_n]
=p^{dk_0}\geq p^{f(j_0+1)},
\]
and each ramification jump contributes an index at most $p^f$ by \eqref{eq:pf}, we obtain $r^n_{s(n)-j_0}\geq\psi_n(e(n-k_0)-c)$, that is

\begin{equation}\label{eq:dis1} \varphi_n(r^n_{s(n)-j_0})\geq e(n-k_0)-c.\end{equation}

Notice that $[G_{n-k_0-1-\lceil \frac {2c}e\rceil}:G_n]\leq p^{A(k_0+1+\lceil \frac {2c}e\rceil)}$, so that there cannot be more than $A(k_0+1+\lceil \frac {2c}e\rceil)$ jumps between $G_n$ and $G_{n-k_0-1-\lceil \frac {2c}e\rceil}$. It follows that \begin{equation}\label{eq:dis2} \varphi_n(r^n_ {s(n)-A(k_0+1+\lceil \frac {2c}e\rceil)})\leq e(n-k_0-1)-c.\end{equation}
Put $j=A(k_0+1+\lceil \frac {2c}e\rceil)-j_0$. 
Then, by \eqref{eq:dis1} and \eqref{eq:dis2},
\[\varphi_n(r^n_{s(n)-j_0})-\varphi_n(r^n_{s(n)-j_0-j})\geq e\]
so that by \eqref{eq:dis0}
\begin{equation*}\frac {r^n_{s(n)-j_0}} {e_n}\geq \frac e {p^{fA(k_0+1+\lceil \frac {2c}e\rceil)}}. \qedhere\end{equation*}
\end{proof}

As an immediate consequence we obtain the following, which is proved as in \cite[Proposition 1.3]{ContiTerracini2025}:

\begin{proposition}
\label{prop:indexbound}
For every sufficiently large $n$, there exists a positive
integer $t_n$ such that $Q_n[t_n]$ contains $G_{n-1}/G_n$
and $e_n/t_n$ is bounded above independently of $n$. 
\end{proposition}

\begin{proof}
For $n$ sufficiently large, and with the notation we introduced earlier, consider the tower of groups
\[G_n\subsetneq G_{n-1}\subseteq G(e(n-1)-c)\subseteq G_{n-1-\lceil\frac {2c} e\rceil }.\] There are at most $A(1+\lceil\frac {2c} e\rceil)$  jumps between $G_n$ and $G_{n-1-\lceil\frac {2c} e\rceil}$. Then $G(e(n-1)-c)\subseteq G(\varphi_n(r^n_{s(n)-A(1+\lceil\frac {2c} e\rceil)}))$, so that \[ G_{n-1}\subseteq G(\varphi_n(r^n_{s(n)-A(1+\lceil\frac {2c} e\rceil)})). \] Put  $t_n=r^n_{s(n)-j_0}$, with $j_0=A(1+\lceil\frac {2c} e\rceil)$. Taking quotients by $G_n$ we obtain
\[G_{n-1}/G_n \subseteq Q_n[t_n].\] Then the claim follows by applying Proposition \ref{teo:bound1}, thanks to the fact that the constant $C$ only depends on $j_0$, which is independent of $n$.
\end{proof}

\section{Property (B) for towers of fields}\label{sect:towers}

Recall that we fixed an embedding $\Qbar \hookrightarrow \Qbar_p$. For algebraic extensions $\KK,\FF,\LL,\ldots$ of $\QQ$, we denote by $K,F,L,\ldots$ their topological closures in $\Qbar_p$.

Throughout this section, let $\KK$ be a number field and $v$ the place of $\KK$ determined by the embedding $\Qbar \hookrightarrow \Qbar_p$.

Let
\begin{equation}\label{eq:tower} \KK\subseteq  \LL_0\subseteq \LL_1\subseteq \ldots \subseteq \LL=\bigcup_n \LL_n\end{equation}
be a tower of algebraic extensions such that $\LL_n/\KK$ is Galois for every $n$, $\GG=\Gal(\LL/\LL_0)$ is a $p$-adic Lie group of positive dimension, and $\GG_n=\Gal(\LL/\LL_n)$ is a Lie filtration of $\GG$.

\begin{definition}
    \label{ass:tower}
We introduce the following conditions on the tower \eqref{eq:tower}:

\begin{description}
\item[\rm{(DVF)}]\label{ass1:DVF} $L_0$ is discretely valued i.e. the ramification index $e(L_0/\Q_p)$ is finite.
\item[\rm{(LTR)}] The local extension $L/L_0$ is totally ramified. 
\item[\rm{(NC)}] \label{ass1:NC} Put $G_n=\GG_n\cap\Gal(L/K)$. For every $n$, the normal closure of $G_n/G_{n+1}$ in $\Gal(\LL_{n+1}/\KK)$ 
is $\GG_n/\GG_{n+1}$.
\item[\rm{(CE)}] $\LL/\KK$ has the \emph{central element property}: there exists $\tau \in Z(\Gal(\LL/\KK))$,  such that  \begin{equation*}\hbox{for every  } \zeta\in\mug_{p^\infty}(\LL), \quad\tau(\zeta)=\zeta^g\quad\hbox{with } g\in\ZZ,\  g>1.\end{equation*}
\end{description}
\end{definition}

\begin{theorem}\label{thm:Btower}
If conditions $\mathrm{(DVF) ,(LTR), (NC)}$ hold for the tower \eqref{eq:tower}, then the set 
\begin{equation*}\label{eq:Z-tower} \mathcal{Z}\coloneqq  \bigcup_{n\ge 1}\{\alpha\in\LL_n^\times \ |\ \exists \sigma\in \GG_{n-1} \hbox{ such that } \sigma(\alpha)/\alpha\not\in\mug_{p^\infty}\} \end{equation*} 
has property \B. 

\noindent If, in addition, $\mathrm{(CE)}$ holds and $\LL_0$ has property \B,
then $\LL$ has property \B.
\end{theorem}

\begin{proof} 
 Let $\alpha\in\mathcal{Z}\setminus \mug$. Then there exists a unique level $n\geq 1$ such that $\alpha\in\LL_n$ and $\sigma(\alpha)/\alpha\not\in\mug_{p^\infty}$ for some   $\sigma\in \GG_{n-1}$. 
By the minimality of $n$, we have $\alpha\not\in \LL_{n-1}$.  
We claim that (NC) implies the existence of $\sigma\in  G_{n-1}$ and $\eta\in\Gal(\LL/\KK)$ such that
\[
\frac{\sigma\eta(\alpha)}{\eta(\alpha)}\not\in \mug_{p^\infty}.
\]
Otherwise for every $\sigma\in  G_{n-1}$ and $\eta\in\Gal(\LL/\KK)$, $\eta^{-1}\sigma\eta$ would fix a $p$-power of $\alpha$, and by (NC), every element of  $\GG_{n-1}$ would fix a suitable $p$-power of $\alpha$, in contradiction with the hypothesis that $\alpha\in\mathcal{Z}$. Therefore, up to replacing $\alpha$ by $\eta(\alpha)$ for a suitable $\eta\in\Gal(\LL/\KK)$, which does not affect the height, we can find $\sigma\in G_{n-1}$ that satisfies
\begin{equation}
\label{eq:tecnica}
\sigma(\alpha)^{p^\lambda}\neq\alpha^{p^\lambda} \quad \text{for any  integer } \lambda>0.
\end{equation}
By $\mathrm{(LTR),(DVF)}$ and Proposition \ref{prop:indexbound}, there exists an integer $t_n$ with the property that $\Gal(L_n/L_{n-1}) \subseteq \Gal(L_n/L_{0})[t_n]$ and $\frac{t_n}{e(L_n/L_0)}$ is bounded below by a positive constant for $n\to\infty$. By the assumption (DVF)  there exists a constant $C'>0$ (independent of $n$) such that
\begin{equation*}\label{eq:seninproof}  \frac{t_n+1}{e(L_n/\QQ_p)}\geq C'. \end{equation*}
Since $\sigma\in \Gal(L_n/L_{0})[t_n]$, for any $\gamma\in\calo_ {\LL_n}$ we have
$$
\vert\sigma(\gamma)-\gamma\vert_{v}\leq p^{-(t_n+1)/{ e(L_n/\QQ_p)}}\leq p^{-C'}
$$
where $v$ is the place of $\LL$ corresponding to the embedding $\LL\hookrightarrow L$. Choose $T$ such that $p^T\geq 
 \frac {p^A [\KK:\QQ]}{[K:\QQ_p]}$, where $p^A= \sup_{k}[\LL_k:\LL_{k-1}]<\infty$.
By~\cite[Lemma 2.1]{AmorosoDavidZannier2014} (note that this lemma holds for any extension of $\QQ_p$, not only for number fields), there exists a positive integer $\lambda$ which is explicitly bounded in terms of $p$, $C'$ and $T$, such that
\begin{equation}
\label{metric1}
\vert\sigma(\gamma^{p^\lambda})-\gamma^{p^\lambda}\vert_v\leq p^{-{p^{T+1}}},
\end{equation}
for every $\gamma\in\calo_ {\LL_n}$. \\
Choose a finite Galois extension $\MM/\KK$ that contains $\alpha$ and such that $\MM\subseteq \LL_n$; for example, take $\MM$ to be the Galois closure of $\KK(\alpha)$.
Then $[\MM:\MM\cap \LL_{n-1}]=[\MM\LL_{n-1}:\LL_{n-1}]\leq [\LL_n:\LL_{n-1}]\leq p^A$.  
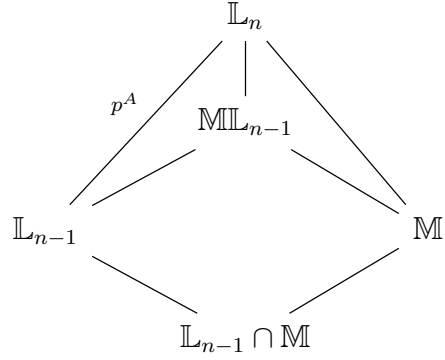
\begin{figure}[h]
    \centering
    \begin{tikzcd}
    	{} & {\mathbb{L}_n} & \\
    	{} & {\mathbb{M}\mathbb{L}_{n-1}} & {} \\
    	{\mathbb{L}_{n-1}} & {} & {\mathbb{M}} \\
    	& {\mathbb{L}_{n-1}\cap \mathbb{M}}
    	\arrow[no head, from=1-2, to=2-2]
    	\arrow["{\leq p^A}"',  no head, from=1-2, to=3-1]
        \arrow[no head, from=1-2, to=3-3]
    	\arrow[no head, from=2-2, to=3-1]
    	\arrow[no head, from=2-2, to=3-3]
    	\arrow[no head, from=3-1, to=4-2]
    	\arrow[no head, from=3-3, to=4-2]
    \end{tikzcd}
\end{figure}\\
We consider the action of $\Gal(\MM/\KK)$ on itself by conjugation and the stabilizer $$C(\tilde\sigma)=\{\nu\in\Gal(\MM/\KK)\;\vert\; \nu\tilde\sigma\nu^{-1}=\tilde\sigma\},$$
 where $\tilde\sigma$ is the restriction of $\sigma$ to $\MM$. Let $\nu\in C(\tilde\sigma)$. We use~\eqref{metric1} with $\nu^{-1}\gamma$ instead of $\gamma$. Since $\tilde\sigma\nu^{-1}=\nu^{-1}\tilde\sigma$ and $\vert\nu^{-1}(\star)\vert_{\tilde v}=\vert\star\vert_{\nu(\tilde v)}$ for the restriction $\tilde v$ of $v$ to $\MM$, we get, for every $\gamma\in\calo_\MM$,  
$$
\vert \sigma(\gamma^{p^\lambda})-\gamma^{p^\lambda}\vert_w\leq p^{-p^{T+1}}
$$
with $w=\nu(\tilde v)$ and thus for any $w$ in the orbit $S$ of $\tilde v$ under the action of $C(\tilde \sigma)$ on the places of $\MM$. 
Using~\cite[Lemma 1]{AmorosoDvornicich2000} as in the proof of~\cite[Lemma 2.2]{AmorosoDavidZannier2014} we find
$$
\vert\sigma(\alpha^{p^\lambda})-\alpha^{p^\lambda}\vert_w
\leq c(w)\max(1,\vert\sigma(\alpha)\vert_w)^{p^\lambda}\max(1,\vert\alpha\vert_w)^{p^\lambda},\; \forall w\in S
$$
with $c(w)=p^{-p^{T+1}}$. This last inequality also holds for an arbitrary place $w$ of $\MM$ not in $S$, with 
$$
c(w)=
\begin{cases}
1,& \hbox{if } w\nmid\infty;\;\\
2,& \hbox{if } w\mid \infty.
\end{cases}
$$
Let $d= [\MM:\Q]$ and $d_{w}=[\MM_w:\QQ_{\lvert w\rvert}]$ for every place $w$ of $\MM$ above a rational place $\lvert w\rvert$. Using the product formula on $\sigma(\alpha)^{p^\lambda}-\alpha^{p^\lambda}$, which is nonzero by~\eqref{eq:tecnica}, as in the proof of~\cite[Lemma 2.2]{AmorosoDavidZannier2014} we get: 
\begin{align*}
0 & = \sum_w\frac{d_{w}}{d}\log\vert \sigma(\alpha^{p^\lambda})-\alpha^{p^\lambda}\vert_w\\
 & \leq \sum_w\frac{d_{w}}{d}\left(\log c(w)
 +p^\lambda\log\max\{1,\vert\sigma(\alpha)\vert_w\}+p^\lambda\log\max\{1,\vert\alpha\vert_w\}\right)\\
& = \left(\sum_{w\mid\infty}\frac{d_{w}}{d}\right)\log 2
-\left(\sum_{w\in S}\frac{d_{w}}{d}\right)p^{T+1}\log p + 2p^\lambda h(\alpha).
\end{align*}
Since $d_{w}=d_{\tilde v}$ for any place  $w\in S$ and $\sum_{w\mid\infty}\frac{d_{w}}{d}=1$, we get
\begin{equation}
\label{first-bound}
2p^\lambda h(\alpha)\geq \frac{d_{\tilde v}}{d}\vert S\vert\cdot p^{T+1}\log p-\log 2.
\end{equation}
To conclude, we need a lower bound for $|S|$. Since $\tilde\sigma$ belongs to the normal subgroup
$\Gal(\MM/\MM\cap\LL_{n-1})$ of $\Gal(\MM/\KK)$,
its conjugacy class $O$ is contained in  $\Gal(\MM/\MM\cap \LL_{n-1})$, so that
$$
\vert C(\tilde\sigma)\vert =\frac{[\MM:\KK]}{\vert O\vert} \geq \frac{[\MM:\KK]}{[\MM:\MM\cap\LL_{n-1}]}\geq \frac{[\MM:\KK]}{[\LL_n:\LL_{n-1}]}\geq \frac {[\MM:\KK]}{p^{A}}.
$$
The stabilizer of $\tilde v$ under the action of $\Gal(\MM/\KK)$ is by definition the decomposition subgroup of $\Gal(\MM/\KK)$ at $\tilde v$, which is isomorphic to the local Galois group $\Gal(M/K)$; its cardinality is $[M:K]$.   Therefore,
$$
\vert S\vert 
=\frac{\vert C(\tilde\sigma)\vert}{\vert \mathrm{Stab}(\tilde v)\cap C(\tilde\sigma)\vert}
\geq \frac{[\MM:\KK]}{p^A[M:K]}=\frac{[K:\QQ_p]}{p^A[\KK:\QQ]}\frac {d}{d_{\tilde v}}\geq \frac 1 {p^T}\frac {d}{d_{\tilde v}}.
$$
Finally, from~\eqref{first-bound} we obtain %$h(\alpha)\geq c$ with
\begin{align}\label{eq:lambda}
 h(\alpha)\geq \frac{\log(p^p/2)}{2p^{\lambda}}>0, &   
\end{align}
hence the first part of the theorem.\\
For the second part, we use the following lemma, which is a straightforward generalization of an argument already used in the proof of \cite[Proposition 6.1]{Habegger2013}: 
\begin{lemma}\label{lem:beta}
Assume that $\LL/\KK$ satisfies condition \textup{(CE)}, and let $\tau$ and $g$ be as in Definition \ref{ass:tower}.
For every $\alpha\in\LL^\times\setminus\mug$, the element
$\beta=\tau(\alpha)/\alpha^g$ belongs to
$\LL^\times\setminus\mug$ and satisfies
\[
h(\beta)\leq(g+1)h(\alpha).
\]
Moreover, for every $\sigma\in\Gal(\LL/\KK)$,
\[
\frac{\sigma(\beta)}{\beta}\in\mug_{p^\infty}
\quad\Longrightarrow\quad
\sigma(\beta)=\beta.
\]
\end{lemma}
\begin{proof} Define $\beta=\tau(\alpha)/\alpha^g$. Observe that if $\beta$ were a root of unity, then we would have
\begin{equation*} h(\alpha)=h(\tau(\alpha))=h(\beta\alpha^g)=h(\alpha^g)=gh(\alpha)\end{equation*}
and since $g>1$ this would imply that $h(\alpha)=0$, which is a contradiction since $\alpha$ is not a root of unity. Observe also that 
\[h(\beta)=h(\tau(\alpha)\alpha^{-g})\leq h(\tau(\alpha))+gh(\alpha)=(g+1)h(\alpha).\]
This means that $h(\alpha)\geq \frac{h(\beta)}{g+1}$. Let $\sigma\in\Gal(\LL/\KK)$; if $\sigma(\beta)\not =\beta$, then, using the fact that $\tau$ is central, we get
\[\frac{\sigma(\beta)}\beta =\frac {\sigma\tau(\alpha)}{\tau(\alpha)}\left(\frac {\sigma(\alpha^g)}{\alpha^g}\right)^{-1}=\frac {\tau(\eta)}{\eta^g},\]
where $\eta=\sigma(\alpha)/\alpha\in\LL^\times$. Now if $\sigma(\beta)/\beta$ were a root of unity of order a power of $p$, then its height would be zero and since $g>1$, $\eta$ would be a root of unity since $h(\eta)=0$. Writing $\eta=\eta_1\eta_2$ with $\eta_1\in\mug_{p^\infty}$ and $\eta_2$ of order not divisible by $p$, we  have $\tau(\eta_1)/\eta_1^g=1$ by property (CE). This implies that $\sigma(\beta)/\beta=\tau(\eta_2)/\eta_2^g$ has order not divisible by $p$, a contradiction. This shows that $\sigma(\beta)/\beta\not\in\mug_{p^\infty}$
\end{proof}
We now come back to the proof of Theorem \ref{thm:Btower}. Assume that \textup{(CE)} holds and that $\LL_0$ has property \B. Let $\alpha\in\LL^\times\setminus\mug$, and let $\beta$ be given by Lemma~\ref{lem:beta}. If $\beta\notin\LL_0$, choose the least $n\geq1$ such that $\beta\in\LL_n$. There exists $\sigma\in\GG_{n-1}$ that does not fix $\beta$, and Lemma \ref{lem:beta} implies that $\sigma(\beta)/\beta\notin\mug_{p^\infty}$. Hence $\beta\in\mathcal Z$. Let $c_0>0$ be a lower bound for the height on $\LL_0^\times\setminus\mug$, and let $c_{\mathcal Z}>0$ be the lower bound obtained in the first part. Then
\[
h(\alpha)\geq\frac{h(\beta)}{g+1}
\geq\frac{\min\{c_0,c_{\mathcal Z}\}}{g+1}>0.
\]
This proves that $\LL$ has property \B.\end{proof}

\begin{remark}\phantom{a}
Note that the hypotheses of Theorem \ref{thm:Btower} rest on the existence of a suitable Lie filtration $\Gal(\LL/\LL_n)$ of the $p$-adic Lie group $\Gal(\LL/\LL_0)$. In our applications, there will be a favourite choice of filtration.
    \end{remark}

\section{Property \B for the compositum of fields} \label{sect:compositum}
In this section, we adapt the results of the previous section to the setting of the compositum of two Galois extensions of a number field: one with finite ramification above a fixed prime and satisfying a lower bound for the height, and the other a potentially totally ramified $p$-adic Lie extension.

As in the previous section, we denote by $K,L,F,\ldots $ the closures in $\ovl\Q_p$ of algebraic extensions $\KK,\LL,\FF,\ldots $ of $\QQ$, via our fixed embedding $\ovl\Q\hookrightarrow\ovl\Q_p$. 
Throughout this section, let $\KK$ be a number field and $v$ the place of $\KK$ determined by the above embedding. 

For a Galois extension $L/K$ contained in $\ovl\QQ_p$, we denote by $I(L/K)$ the corresponding inertia subgroup of $\Gal(L/K)$.

Suppose we are given a tower of fields  \begin{equation}\label{eq:tower2} \KK\subseteq \LL_0\subseteq \LL_1\subseteq \ldots \subseteq \LL=\bigcup_n \LL_n,\end{equation} such that $\LL/\KK$ is Galois, $\GG=\Gal(\LL/\KK)$ is a $p$-adic Lie group of positive dimension, and $\GG_n=\Gal(\LL/\LL_n)$ is a Lie filtration of $\GG$ (note that this is slightly different from the filtration \eqref{eq:tower}, where instead $\Gal(\mathbb{L}/\mathbb{L}_0)$ was the $p$-adic Lie group).

Let $\FF/\KK$ be a Galois extension and consider the tower of composita  
$$\KK\subseteq  \FF \subseteq \FF\LL_0\subseteq \FF\LL_1\subseteq \ldots \subseteq \FF \LL=\tilde \LL.$$
Since each $\GG_n$ is a normal subgroup of $\GG$, each intermediate field $\FF\LL_n$ is Galois over $\KK$. We show that $\Gal(\widetilde \LL/\FF)$ is a $p$-adic Lie group. We have
\begin{equation}\label{cartan}
    \Gal(\widetilde\LL/\mathbb{F})\xrightarrow{\sim} \Gal(\mathbb{L}/\mathbb{L}\cap\mathbb{F})\hookrightarrow \Gal(\mathbb{L}/\mathbb{K})
\end{equation}
so $\Gal(\widetilde\LL/\mathbb{F})$ is isomorphic to a closed subgroup of $\Gal(\mathbb{L}/\mathbb{K})$. By Cartan's theorem on closed subgroups \cite[Chapter V, \S 9]{SerreLie2006}, $\Gal(\mathbb{L}/\mathbb{L}\cap\mathbb{F})$ is a $p$-adic Lie subgroup of $\Gal(\mathbb{L}/\mathbb{K})$, so $\Gal(\widetilde \LL/\mathbb{F})$ has a unique structure as a $p$-adic Lie group making the isomorphism of topological groups in \eqref{cartan} an isomorphism of $p$-adic Lie groups. 

We show that the groups $\Gal(\wtl\LL/\FF\LL_n)$ define a Lie filtration on $\Gal(\wtl\LL/\FF)$, by pulling back the Lie filtration on $\Gal(\LL/\KK)$ to $\Gal(\wtl \LL/\FF)$ in the following way. Under the restriction isomorphism in \eqref{cartan}, the subgroup $\Gal(\wtl\LL/\FF\LL_n)$ corresponds to
\[
    \Gal(\LL/\LL\cap\FF)\cap\GG_n =\Gal(\LL/\LL_n(\LL\cap\FF)).
\]
Thus, intersecting the Lie filtration $(\GG_n)$ with the closed subgroup $\Gal(\LL/\LL\cap\FF)$ and transporting it through \eqref{cartan}, we obtain the Lie filtration $(\Gal(\wtl\LL/\FF\LL_n))$ on $\Gal(\wtl\LL/\FF)$.

We introduce the following variant of the central element property:

\begin{definition}\label{def:ICE} 
    We say that a Galois extension $\LL/\KK$ has the \emph{inertial central element property} \rm{(ICE)} at $v$ if there exists an element $\tau \in I(\Qbar_p/K)$,  such that $\tau|_\LL\in Z(\GG)$ and \begin{equation}\label{eq:H1inertial}\hbox{for every  } \zeta\in\mug_{p^\infty}, \quad\tau(\zeta)=\zeta^g\quad\hbox{with } g\in\ZZ,\  g>1.\end{equation}
\end{definition}

\begin{proposition}\label{prop:compositum}
Assume that conditions {\rm(LTR)}, {\rm (NC)} in Definition \ref{ass:tower} are satisfied for the tower \eqref{eq:tower2}.\\ Let $\FF/\KK$ be a Galois extension  whose local closure $F$ is discretely valued. \\
Then
\begin{itemize}
\item[(i)] The $p$-adic Lie group $\Gal(\wtl \LL/\FF)$ has positive dimension. 
\item[(ii)] For a suitable $n_0\geq 0$, the sequence  
\begin{equation}\label{eq:tower3}
    \KK\subseteq \widetilde\LL_0\subseteq \widetilde\LL_1 \subseteq \cdots \subseteq \widetilde{\LL} =\bigcup_{n}\wtl\LL_n
\end{equation}
obtained by  setting $\widetilde \LL_i=\FF\LL_{n_0+i}$, satisfies conditions {\rm (DVF)}, {\rm (LTR)}, {\rm (NC)}.
\item[(iii)] If in addition $\LL/\KK$ satisfies condition  $\mathrm{(ICE)}$ in Definition \ref{def:ICE}, then $\widetilde \LL/\KK$ satisfies condition $\mathrm{(CE)}$.
\end{itemize}
\end{proposition}

\begin{proof}
\begin{itemize}
\item[$(i)$] 
        Assume by contradiction that $\Gal(\widetilde\LL/\mathbb{F})\simeq\Gal(\mathbb{L}/\mathbb{L}\cap \mathbb{F})$ has dimension 0, i.e. that $\mathbb{L}/\mathbb{L}\cap \mathbb{F}$ is finite.  Since $e(F/K)$ is finite, so is $e((L\cap F)/K)$, so that $e(L/K)$ is also finite. By assumption, the extension $\LL/\LL_0$ satisfies (LTR), so $I(L/K)$ is both finite and open in the decomposition group. We deduce that $\Gal(L/K)$ is finite, hence $G_n=G_{n+1}$ for all sufficiently large $n$, which implies  $\mathbb{G}_n=\mathbb{G}_{n+1}$ by assumption \rm{(NC)}. This is only possible if $\Gal(\mathbb{L}/\mathbb{K})$ is finite, contradicting the assumption that it is positive dimensional. 
\item[$(ii)$] For any finite extensions $F^\prime /K$ and $L^\prime/K$ contained in $F$ and $L$ respectively we have the following chain of inequalities involving residue degrees $f$ and ramification indices $e$:
        \begin{align*}
            f(F^\prime L^\prime/F^\prime)& =\frac{f(F^\prime L^\prime/K)}{f(F^\prime/K)}=\frac{f(F^\prime L^\prime/L^\prime)f(L'/K)e(F^\prime /K)}{[F^\prime:K]}\\
            &\leq f(L^\prime/K)e(F^\prime/K)\frac{[F^\prime L^\prime:L^\prime]}{[F^\prime:K]}\leq f(L^\prime/K)e(F^\prime/K).
        \end{align*}
        Taking the limit over all $F^\prime$ and $L^\prime$ as above, using that $f(L/K)$ is finite by the hypothesis (LTR) and $e(F/K)$ is finite by the hypothesis on $\FF$, we deduce that $f(FL/F)\leq f(L/K)e(F/K)<\infty$. In particular $FL/F$ is potentially totally ramified; since the groups $\Gal(\widetilde L/FL_n)$ form a neighborhood basis of the identity  in $\Gal(\widetilde L/F)$,  there exists an index $n_1$ such that $FL/FL_{n_1}$ is totally ramified. \\
         Since $F$ is discretely valued and $L/L_0$ is totally ramified,  $L\cap F$ is finite over $K$. Therefore, $L\cap F\subseteq L_{n_0}$ for some $n_0\geq n_1$.
         We put $\wtl\LL_n= \FF\LL_{n+n_0}$ for every $n\in\NN$. 
 Observe that for every $n$, the extension $\FF\LL_n/\FF$ is finite, so $FL_n$ is discretely valued. Therefore, conditions (DVF) and (LTR) are satisfied for the tower \eqref{eq:tower3}.\\
We now verify condition (NC). Since $L\cap F\subseteq L_n$ for every $n\geq n_0$, we have \begin{equation}\label{eq:iso}
    \text{  $\forall n\geq n_0$,  $\mathrm{Res}:\Gal(F L_{n+1}/F L_{n})\rightarrow \Gal( L_{n+1}/ L_{n})$ is an isomorphism.}
\end{equation}
Indeed, the image of the restriction is $\Gal(L_{n+1}/L_{n+1}\cap L_{n}F)\subseteq \Gal(L_{n+1}/L_{n})$, and by Dedekind's modular law
$$L_{n+1}\cap L_{n}F=(L_{n+1}\cap F)L_{n}.$$
But for $n\geq n_0$, $L_{n+1}\cap F=L_{n}\cap F$, so $$(L_{n+1}\cap F)L_{n}=(L_{n}\cap F)L_{n}=L_{n}.$$

The restriction map
\[
\mathrm{Res}:\Gal(\FF\LL_{n+1}/\KK)\longrightarrow
\Gal(\LL_{n+1}/\KK)
\]
is surjective, and it is injective when restricted to $B=\Gal(\FF\LL_{n+1}/\FF\LL_n)$.  Let $N$ be the normal closure of the subgroup $\Gal(FL_{n+1}/FL_n)$ in $\Gal(\FF\LL_{n+1}/\KK)$. Then $N\subseteq B$, since the latter  is normal in $\Gal(\FF\LL_{n+1}/\KK)$. By \eqref{eq:iso} and condition \textup{(NC)} for the tower \eqref{eq:tower2},
\(
\mathrm{Res}(N)=\Gal(\LL_{n+1}/\LL_n)
\).
Since $\mathrm{Res}(B)\subseteq\Gal(\LL_{n+1}/\LL_n)$ and $\mathrm{Res}$ is injective on $B$, we obtain $N=B$, proving \textup{(NC)} for the tower \eqref{eq:tower3}.

    \item[$(iii)$]  Assume  now that (ICE) holds for $\LL/\KK$; notice that the restriction $\tau_{\mid \FF}$ belongs to $I(F/K)$, which has finite order $r$ by hypothesis, so $\tau_{\mid \FF}^r=\id_{\mid \FF}$. Since $\tau^r_{\mid \LL}\in Z(\Gal(\LL/\KK))$ we have that $\tau^r|_{\wtl\LL}\in Z(\Gal(\widetilde \LL/\KK))$. Moreover $\tau^r$ satisfies the condition on the roots of unity \eqref{eq:H1inertial} since $\tau$ does, so $\widetilde \LL/\KK$ satisfies condition (CE) in Definition \ref{ass:tower}.
\end{itemize}
\end{proof}

We deduce the following:

\begin{theorem}\label{cor:composto2} Assume that conditions {\rm(LTR)}, {\rm (NC)} in Definition \ref{ass:tower} are satisfied for the tower \eqref{eq:tower2}.\\
Let $\FF/\KK$ be a Galois extension having a discrete valuation at (a place over) $v$.\\
  Let $n_0$ be as in Proposition \ref{prop:compositum}, and put
 $\widetilde \LL_n=\LL_{n_0+n}\FF$, $\wtl\GG_n=\Gal(\widetilde\LL/\widetilde \LL_n)$,
 $\widetilde \LL=\LL\FF$. Then:
\begin{itemize}
\item[(i)] The set 
\begin{equation*}
\widetilde{\mathcal{Z}}\coloneqq  \bigcup_{n\ge 1}\{\alpha\in\widetilde\LL_n ^\times \ |\ \exists \sigma\in \widetilde\GG_{n-1} \hbox{ such that } \sigma(\alpha)/\alpha\not\in\mug_{p^\infty}\} \end{equation*} 
has property \B. 
\item[(ii)] If in addition $\FF$ has property \SB and $\LL/\KK$ satisfies  condition {\rm (ICE)},   then $\widetilde \LL$ has property \SB.
\end{itemize}
\end{theorem}
\begin{proof} Statement ($i$) follows from
Proposition~\ref{prop:compositum} and
Theorem~\ref{thm:Btower}.

Under the assumptions of ($ii$), the field
$\widetilde\LL_0=\FF\LL_{n_0}$ has property \B,
since it is a finite extension of $\FF$ and $\FF$ has
property \SB. Thus the same proposition and theorem imply
that $\widetilde\LL$ has property \B. 

To prove property \SB in ($ii$), notice that every finite extension of $\widetilde \LL$ can be written as $\LL\FF'$ for some finite extension $\FF'/\FF$. Since $\FF/\KK$ is Galois, the Galois closure $\FF''$  of $\FF'$ over $\KK$ is finite over $\FF'$. By replacing $\FF$ by $\FF''$ in the statement of the theorem, we obtain property \B  for the extension $\LL\FF''$.
\end{proof}

\section{Supercuspidal forms: the compatible system outside $p$}\label{sec:outsidep}

Let $p$ be an \textit{odd} prime. In this section and the following, we apply the results we proved on property \B to the case of the field cut out by the adelic Galois representation attached to a modular Hecke eigenform, supercuspidal at $p$.

\subsection*{The Hecke correspondence}
Let $\pi_p$ be an irreducible admissible representation of $\GL_2(\QQ_p)$, with coefficients in an arbitrary algebraically closed field of characteristic zero. As explained in \cite[\S 0.5]{Carayol1986}, by (a variant of) the local Langlands correspondence, there is a corresponding two-dimensional Frobenius semisimple representation $\sigma(\pi_p)$ of the Weil-Deligne group  $WD(\QQ_p)$.
This so-called ``Hecke'' correspondence satisfies the following conditions:
\begin{itemize}
\item[a)] The determinant of $\sigma(\pi_p)$ is related to the central character $\psi_p$ of $\pi_p$ by
$$\det(\sigma(\pi_p))=\psi_p^{-1}\omega_p^{-1}$$
where $\omega_p=|\cdot |_p$.
\item[b)] $\sigma(\psi\cdot \pi_p)=\psi^{-1}\sigma(\pi_p)$ for every character $\psi$ of $\QQ_p^\times$.
\end{itemize}
Note that in \cite[\S 0.4]{Carayol1986} the local class field
theory isomorphism is normalized so that geometric Frobenius
corresponds to a uniformizer. The same symbol denotes a character
of $\QQ_p^\times$ and the corresponding character of $W_{\QQ_p}$.
Moreover, the Hecke correspondence is invariant under automorphisms
of the coefficient field.

\subsection*{A compatible system of representations of $W_{\QQ_p}$}\label{subsect:compsyst} Let $f\in S_k(\Gamma_0(N),\chi)$ be a new Hecke eigenform for some $k\geq 2$ and $N\ge 1$. Let $\pi_f$ be the corresponding automorphic representation of $\GL_2(\AA_\QQ)$; we use here Carayol's normalisation.
%We follow the normalisation in \cite[\S 2.1]{LoefflerWeinstein2012}.
This representation decomposes
as a restricted tensor product over all places $v$ of $\QQ$:
\[ \pi_f = \bigotimes_v \pi_{f,v}, \]
where $\pi_{f,v}$ is an irreducible admissible representation of $\GL_2(\QQ_v)$.
By \cite[0.6]{Carayol1986}, there exists a finite extension $\KK_f$  of $\QQ$ such that for every finite prime $p$, the representation $\sigma(\pi_{f,p})$ is defined over $\KK_f$.
Then, for every finite place $\lambda$  of $\KK_f$ \emph{not dividing }$p$, there is a \emph{continuous} representation $\sigma^\lambda (\pi_{f,p}): W_{\QQ_p} \to \GL_2(\KK_{f,\lambda})$, giving rise to a compatible system of representations $\{\sigma^\lambda(\pi_{f,p})\}_\lambda$ of $W_{\QQ_p}$. 

The two constructions described above are related by a celebrated theorem of 
 Carayol \cite[Theorem (A)]{Carayol1986}. 
 
\begin{theorem}[Carayol] For every finite place $\lambda$ of $\KK_f$, let $\rho_{f,\lambda}:G_\QQ\longrightarrow \GL_2(\KK_{f,\lambda})$
be the $\lambda$-adic representation of $G_\QQ$ associated to $f$.
If $\lambda$ does not divide $p$, then $\rho_{f,\lambda}$ coincides with $\sigma^\lambda(\pi_{f,p})$ over $W_{\QQ_p}$.
\end{theorem}
Carayol's theorem yields the following consequence for the local image at \(p\) of the prime-to-\(p\) part of the Galois representation attached to a modular form.
\begin{proposition}\label{prop:primetop}
Let $f\in S_k(\Gamma_0(N),\psi)$ be an eigenform, and let
\[\rho_f^{(p)}:G_\QQ \longrightarrow \GL_2\Big(\prod_{\lambda \nmid p} \KK_{f,\lambda}\Big)\] 
be the prime-to-$p$ part of the adelic Galois representation associated to $f$. Assume that either
\begin{itemize}
\item $\pi_{f,p}=\pi(\chi_1,\chi_2)$ is a principal series and the unitary part of one of $\chi_1, \chi_2$ has finite order; or
\item $\pi_{f,p}$ is supercuspidal.
\end{itemize}
 Then:
\begin{itemize}
\item[a)] $\rho_{f}^{(p)}(I_p)$ is finite;
\item[b)] the projective image of $\rho_f^{(p)}|_{G_p}$ is finite.
\end{itemize}
\end{proposition}

\begin{proof} Every continuous finite-dimensional complex representation of a profinite group has finite image by the no-small-subgroup property. By Carayol's theorem, the restrictions to inertia of the representations $\rho_{f,\lambda}$, with $\lambda$ not dividing $p$, arise
from the same representation defined over $\KK_f$.
Hence they have a common kernel of finite index,
which proves $a$).  \\
By \cite[\S 2.2.1]{Tate1979}, every irreducible complex representation of $W_F$, $F$ a $p$-adic field, has the form $\sigma\otimes |\cdot |_p^s$, where $\sigma$ is the restriction of a Galois representation, hence it has finite image (by the argument in $a$).  
If $\pi_{f,p}$ is supercuspidal, then $\sigma(\pi_{f,p})$, the associated representation of $W_{\QQ_p}$, is irreducible \cite[\S 0.9]{Carayol1986}, so that $b$) follows. 
If $\pi_{f,p}\cong \pi(\chi_1,\chi_2)$ is a principal series, then
\[
\sigma(\pi_{f,p})\simeq
(\chi_1^{-1}\oplus\chi_2^{-1})|\cdot|_p^{-1/2}
\]
by \cite[\S 0.5]{Carayol1986}.
By the Ramanujan--Petersson conjecture, proved by Deligne
(see for example \cite[\S 2.2]{Blasius2006}), we can write
$\chi_i=\chi_i'|\cdot|_p^s$ with the same real number $s$
and unitary characters $\chi_i'$.
Thus the projective image of $\sigma(\pi_{f,p})$ is finite
if and only if $\chi_1'/\chi_2'$ has finite order.

Since
\(
  \chi_1'\chi_2' \;=\; \psi_p
\),
where $\psi_p$ is the $p$-component of the finite-order nebentype $\psi$, we have
\[
  \chi_1'/\chi_2' \ \text{has finite order}
  \ \Longleftrightarrow\
  \chi_1' \ \text{and}\ \chi_2' \ \text{have finite order}.
\]
(In particular, $\chi_2'=\psi_p\,(\chi_1')^{-1}$, so one has finite order if and only if both do.)

In both cases, by Carayol's theorem,  the projective image of $\rho_f^{(p)}|_{W_{\QQ_p}}$
is finite. Since $W_{\QQ_p}$ is dense in $G_p$, continuity
implies the same conclusion for $\rho_f^{(p)}|_{G_p}$.
\end{proof}

Notice that when $\pi_{f,p}$ is an unramified principal series, $\chi'_1$ and $\chi'_2$ have finite order if and only if the Satake parameters (i.e. the roots of $X^2- {p^{-\frac{k-1} 2}} {a_p}X  + \psi(p)$) are roots of unity. This case includes the strong supersingular case $a_p=0$ studied in \cite{Habegger2013, AmorosoTerracini2024, Piras2026}.

\begin{corollary}\label{cor:ADZ}
Under the assumptions of Proposition \ref{prop:primetop}, the field $\QQ(\rho_f^{(p)})$ has property \ADZ. %subfield $\QQ(\rho_f^{(p)})$ of $\ovl\QQ$ fixed by $\ker\rho_f^{(p)}$ has property \ADZ.
\end{corollary}
\begin{proof}
By definition, $\rho_f^{(p)}(G_\QQ)\cong\Gal(\QQ(\rho_f^{(p)})/\QQ)$, and since this isomorphism maps the scalars in $\rho_f^{(p)}(G_\QQ)$ into $Z(\Gal(\QQ(\rho_f^{(p)})/\QQ))$, the projective image of $\rho_f^{(p)}(G_\QQ)$ admits $\Gal(\QQ(\rho_f^{(p)})/\QQ)/Z(\Gal(\QQ(\rho_f^{(p)})/\QQ))$ as a quotient. In particular, $b)$ of Proposition \ref{prop:primetop} implies that every decomposition group of $\Gal(\QQ(\rho_f^{(p)})/\QQ)/Z(\Gal(\QQ(\rho_f^{(p)})/\QQ))$ at $p$ is finite, so that \ADZ holds. 
\end{proof}

\section{Supercuspidal forms: the $p$-adic Galois representations}\label{sec:padic}
We write $WD(\pi_{p})$ for the Weil--Deligne representation attached to $\pi_{p}$ via the Hecke correspondence recalled in Section \ref{sec:outsidep}.

For $p$ odd, the complex supercuspidal representations of $\GL_2(\QQ_p)$ are parametrized by  \emph{admissible pairs} $(E,\theta)$, where $E$ is a quadratic extension of $\QQ_p$  and $\theta:E^\times\to\CC^\times$ is a character which does not factor through $N_{E/\QQ_p}$ and, if $E/\QQ_p$ is ramified, whose restriction to the group $U_1(E)$ of principal units does not factor through the norm either; see \cite[Definition 3.8]{LoefflerWeinstein2012} and the related constructions in \cite{Gerardin1978}.
 We normalize the parametrization in such a way that, if $\pi_p$ is associated to the pair $(E,\theta)$, then $$WD(\pi_p)=\mathrm{Ind}_{W_E}^{W_{\QQ_p}}\theta, $$ where $\theta$ is regarded as a character of $W_E$ by local class field theory. 
 
We recall that the conductors of $\pi_p$ and $\theta$ are related by the following \emph{conductor formula}
(see for example \cite[\S 4]{MandalMondal2026})
\begin{equation} \label{eq:conduttore} \mathrm{cond}(\pi_{p})=\begin{cases} 2\,\mathrm{cond}(\theta) &\hbox{if $E/\QQ_p$ is unramified}\\
\mathrm{cond}(\theta)+1 & \hbox{if $E/\QQ_p$ is ramified}\end{cases}.\end{equation}
Let $f\in S_k(\Gamma_0(N))$ be an $N$-new Hecke eigenform, with $k\geq 2$ and $\KK_f=\QQ$. Let $p$ be an odd prime. We assume that the $p$ component $\pi_{f,p}$ of the corresponding automorphic representation of $\GL_2(\AA_\QQ)$ is supercuspidal. Let $(E,\theta)$ be the associated admissible pair. 
In this case, the conductor of $\pi_{f,p}$ is the $p$-adic valuation of $N$.  Moreover, if $E$ is a ramified quadratic extension of $\QQ_p$, then $\theta$ has conductor at least 2. It follows that every such modular form $f$ which is twist-minimal of level $Mp^2$, with $p\nmid M$, is supercuspidal arising from a pair $(\QQ_{p^2},\theta)$.

Let 
$$\rho_{f,p}:\GG_\QQ\longrightarrow \GL_2(\ZZ_p)$$
be the $p$-adic representation associated with $f$. 

By \cite[Corollary B.1.3]{ContiTerracini2025} two possible cases arise: 
\begin{description}%[label=(\Alph*)]
\item[\rm{(open)}] $\rho_{f,p}(I_p)$ is open in $\GL_2(\ZZ_p)$; 
\item[\rm{(ind)}] $\rho_{f,p}|_{G_p}$ is induced from a character
$\psi$ of the absolute Galois group of a quadratic extension
of $\QQ_p$.
\end{description}
In case (open), $\rho_{f,p}$ satisfies all the required assumptions thanks to \cite[Proposition 3.2]{ContiTerracini2025} (see Proposition \ref{prop:padic} below). In case (ind), we prove them in the next subsection.

\subsection{(LPTR) for induced representations}\label{sect:induced}
Fix a prime $p$. We say that a representation $\rho$ of
$\GG_\QQ$ is \emph{locally potentially totally ramified}
(LPTR) (at $p$) if the extension $\QQ(\rho)/\QQ$ cut out by
$\rho$ has finite residue degree at $p$.

Assume $f$ as above and $p\geq 5$. Then we are in the case 3) or 4) of \cite[\S 11, Proposition 1]{FontaineMazur1995}, according to the fact that $\rho_{f,p}|_{G_p}$ becomes  crystalline or not over the field $F_1=\QQ_p(\pi_1)$, with $\pi_1=\sqrt[p-1]{-p}$; if this is not the case, then it becomes crystalline over $F_2=\QQ_{p^2}(\pi_2)$, with $\pi_2^{p+1}=\pi_1$. In both cases, the representation $\rho_{f,p}|_{G_p}$ becomes crystalline over a finite extension $F$ of $E$  and the associated filtered $(\varphi,N,\Gal(F/\QQ_p))$-module has the form $D_{III}(0,k-1,a_1,a_2,i_1,i_2)$ (case 3) or $D_{IV}(0,k-1,d,i_1,i_2,\alpha)$ (case 4) (see \cite[\S 11]{FontaineMazur1995}).

The following proposition shows that the condition on the Hecke field of $f$ forces the pair $(E,\theta)$ to be of a very particular form:
\begin{proposition}\label{prop:r=2}
Assume $p\geq 5$. Let $f\in S_k(\Gamma_0(Np^r))$ ($\gcd(p,N)=1$) be a modular newform, supercuspidal at $p$, such that its Hecke field is $\QQ$. Then $r=2$ and $WD(\pi_{f,p})$ is induced by a  character of $\QQ_{p^2}$.
\end{proposition}
\begin{proof} 
Let $(E,\theta)$ be the admissible pair associated to $\pi_{f,p}$.
By the local Langlands correspondence (in particular the main theorem in \cite{Saito1997}) $WD(\pi_{f,p})\otimes_{\Qbar}\Qbar_p\simeq WD(\rho_{f,p}|_{G_p})$. The representation $\rho_{f,p}|_{G_p}$ becomes crystalline over a finite extension  $F$ of $\QQ_p$ if and only if $WD(\rho_{f,p})|_{I_F}$ is trivial (see for example \cite[Lemma 2.2.2]{GhateMezard2009}, taking into account that $\rho_{f,p}|_{G_p}$ is potentially crystalline).
If $p\geq 5$, then the classification of Fontaine and Mazur mentioned above shows in particular that  every two dimensional potentially crystalline representation with coefficients in $\QQ_p$ becomes crystalline  over a field  $F$ \emph{tamely ramified} over $\QQ_p$ (the ramification index of $F$ is indeed a divisor of $p^2-1$).
Since $WD(\pi_{f,p})\simeq \mathrm{Ind}_{W_E}^{W_{\QQ_p}}\theta$  the preceding discussion yields strong restrictions on the character $\theta$. In particular, $\theta$ must be tamely ramified, that is, trivial on $1+\mathfrak{P}_E$. This is incompatible with the admissibility of $(E,\theta)$
when $E/\QQ_p$ is ramified; on the other hand, if $E=\QQ_{p^2}$ then $\theta$ is trivial on $1+p\ZZ_{p^2}$ and so by the conductor formula \ref{eq:conduttore} the exponent of $p$ in the level of $f$ must be  $2$.
\end{proof}

\begin{remark}
    When $p=3$, the situation is more varied, and there exist cuspidal eigenforms with coefficients in $\QQ$ that are supercuspidal at $3$ and have conductor exponent at \(3\) strictly greater than \(2\). However, we have decided to exclude this case from our analysis to avoid overcomplicating the exposition.
\end{remark}

Assume that we are in case (ind). Then $\rho_{f,p}|_{G_p}\simeq \mathrm{Ind}_{G_{E'}}^{G_p} \psi$, where $E'/\QQ_p$ is a quadratic extension and  $\psi$ is a character of $G_{E'}$, and Mackey's formula gives 
$$\rho_{f,p}|_{G_{E'}}\simeq \psi\oplus \psi'$$
where $\psi'$ is obtained from $\psi$ by conjugation by any element of $G_p\setminus G_{E'}$.
Since $\rho_{f,p}|_{G_p}$ is potentially semistable, so are $\psi$ and $\psi'$, and by functoriality $$WD(\rho_{f,p}|_{G_{E'}})\simeq WD(\psi)\oplus WD(\psi')$$ 
where we denote by $WD(\rho)$ the
Weil--Deligne representation associated to a potentially semistable local Galois representation $\rho$.

Let $K'$ be a $p$-adic field containing $E'$ and the image of $\psi$, which we now regard as a character $G_{E'}\to (K')^\times$. 

For every finite extension $K/\QQ_p$, we denote by $r_K:K^\times \to G_K^{ab}$ the local Artin map with arithmetic normalization, and by $\underline{K}^\times=R_{K/\QQ_p}(\GG_m)$ the $\QQ_p$-torus obtained via Weil restriction of scalars. We can assume that $E'\subseteq F$.  By \cite[Proposition B.4 (i)]{ConradLifting2012}, the fact that $\psi$ becomes crystalline on $G_F$ means that $\psi\circ r_{E'}$ coincides with (the restriction to $\QQ_p$-points of) a $\QQ_p$-homomorphism $\omega: \underline{E'}^\times\to \underline{K'}^\times$ on an open subgroup of $\calo_{E'}^\times$ containing $N_{F/E'}(\calo_F^\times)$. 

We can write $\omega=\tau_1^{n_1}\tau_2^{n_2}$, where $\tau_1,\tau_2:E'\hookrightarrow K'$ are the two embeddings of $E'$ and $n_1,n_2\in\ZZ$. Since $E'/\QQ_p$ is quadratic, hence Galois, $\omega$ takes values in $\underline{E'}^\times\subseteq\underline{K'}^\times$. With the arithmetic normalization of the local Artin map, the exponents $n_1,n_2$ are the negatives of the labelled Hodge--Tate weights of $\psi$ (see \cite[Appendix B]{ConradLifting2012}). Since $\rho_{f,p}|_{G_p}$ is induced from $\psi$ and has Hodge--Tate weights $\{0,k-1\}$, we obtain $\{n_1,n_2\}=\{0,1-k\}$. Thus, for a suitable embedding $\tau:E'\hookrightarrow K'$, we have $\omega=\tau^{1-k}$ and
\begin{equation}\label{eq:suF}
(\psi|_{G_F}\circ r_F)(u) =\tau\bigl(N_{F/E'}(u)\bigr)^{1-k} \qquad\text{for every }u\in\calo_F^\times.
\end{equation}

Let $F_0$ be the maximal unramified subextension of $F/\QQ_p$, and let $a=[F_0:\QQ_p]$. Let $V$ be the underlying $\QQ_p$-representation of the one-dimensional $K'$-representation defined by $\psi|_{G_F}$. By \cite[Proposition B.4(ii)]{ConradLifting2012}, the filtered $\varphi$-module $D_{\mathrm{cris}}(V)$ over $F$ covariantly attached to $V$ is free of rank one over $K'\otimes_{\QQ_p}F_0$, and $\varphi^a$ acts as multiplication by
\[
y=\psi(r_F(\pi_F))^{-1}
  \omega\bigl(N_{F/E'}(\pi_F)\bigr)\in(K')^\times,
\]
where $\pi_F$ is any uniformizer of $F$. Consequently, to prove that $\psi$ is potentially totally ramified, it suffices to show that, for some uniformizer $\pi_F$ of $F$, the element
\[
\psi(r_F(\pi_F))
=y^{-1}\omega\bigl(N_{F/E'}(\pi_F)\bigr)
=y^{-1}\tau\bigl(N_{F/E'}(\pi_F)\bigr)^{1-k}
\]
is a root of unity.

Write $W$ for a 2-dimensional $\Q_p$-vector space equipped with the $G_{\Q_p}$-action given by $\rho_{f,p}\vert_{G_{\Q_p}}$. Let $D=D_{\mathrm{pst}}^*(W)$, using the contravariant convention of \cite{FontaineMazur1995}. We describe $D$ in terms of the classification recalled in the previous subsection. Since
\[
(W\otimes_{\QQ_p}K')|_{G_{E'}}
\simeq \psi\oplus\psi',
\]
functoriality gives, after extending coefficients to $K'$ and restricting the descent data to $\Gal(F/E')$,
\[
(D\otimes_{\QQ_p}K')|_{\Gal(F/E')}
\simeq D_{\mathrm{pst}}^*(\psi)
\oplus D_{\mathrm{pst}}^*(\psi').
\] Since the Hodge--Tate weights of $\rho_{f,p}$ are $(0,k-1)$, the filtration on $D_F=F\otimes_{F_0}D$ is determined by the $F$-line $\Delta=\mathrm{Fil}^1D_F=\mathrm{Fil}^{k-1}D_F$. The preceding decomposition into rank-one modules implies that $\Delta$ is stable under the $F$-linear extension of $\varphi^a$, where $a=[F_0:\QQ_p]$. This excludes the Type III case (note that there is a typo in \cite[\S 11, p.~207]{FontaineMazur1995}: the correct formula is $\Delta=F_1(\pi_1^{i_2}\otimes e_1+ \pi_1^{i_1}\otimes e_2)$). Indeed, stability would imply $a_1^a=a_2^a$, hence $v_p(a_1)=v_p(a_2)$. Since $v_p(a_1a_2)=k-1$ and $a_1,a_2\in\QQ_p^\times$, this is impossible because $k$ is even.   For Type IV, the determinant formula in \cite[Proposition 2]{FontaineMazur1995} gives
\[
d=p^{k-1},\qquad i_2=p-1-i_1.
\]
We take $F=F_2=\QQ_{p^2}(\pi_2)$, where $\pi_2^{p^2-1}=-p$. Thus $F_0=\QQ_{p^2}$, and $\varphi^2$ acts on the contravariant module by the scalar $-d$. On the corresponding covariant module, it therefore acts by $(-d)^{-1}=-p^{1-k}$.

The field $F$ contains all three quadratic extensions of $\QQ_p$, so in particular $E'\subseteq F$. The preceding formula gives
\[
\psi(r_F(\pi_2))
=-p^{k-1}\tau\bigl(N_{F/E'}(\pi_2)\bigr)^{1-k}.
\]
Since $[F:E']=p^2-1$ and $\pi_2^{p^2-1}=-p$, we have
\[
\left(\frac{N_{F/E'}(\pi_2)}{p}\right)^{p^2-1}=1.
\]
Hence $\psi(r_F(\pi_2))$ is a root of unity, which proves that $\psi$ is potentially totally ramified.

The above discussion gives the following.

\begin{proposition}\label{prop:indLPTR}
Assume $p\geq 5$. In case {\rm (ind)}, the representation $\rho_{f,p}$ is locally potentially totally ramified at $p$.
\end{proposition}

Note that if we are in case (ind) and $\rho_{f,p}$ is potentially totally ramified (not just locally), then $\rho_{f,p}$ is necessarily (globally) induced: otherwise, the global image would contain an open subgroup of $\SL_2(\ZZ_p)$ (cf. \cite[Corollary B.1.3]{ContiTerracini2025} or \cite[Theorem 2.3.1]{LoefflerAdelic2017}) so that the local image, being abelian up to index 2, would have infinite index in the global one, contradicting the fact that an inertia subgroup is open in the global image. If $\rho_{f,p}$ is globally induced, then $f$ is a CM eigenform, and property \B is already known in this case.

\subsection{Properties of $\rho_{f,p}$}
The $p$-adic Lie group $\GL_2(\ZZ_p)$ is endowed with a canonical Lie filtration by the subgroups $1+p^{n+1}\M_2(\ZZ_p)$, for $n\geq 0$. Intersecting with the image of $\rho_{f,p}$  we obtain a Lie filtration $(\GG_n)$ of its closed subgroup $\GG= \Gal(\QQ(\rho_{f,p})/\QQ)$. Notice that $\GG$ has positive dimension, since $k\geq 2$ and we put
$$\LL=\QQ(\rho_{f,p}), \quad\hbox{ and } \quad \LL_n=\LL^{\GG_n},$$
getting a tower of fields
\begin{equation}\label{eq:tower4} \QQ\subseteq \LL_0\subseteq \LL_1\subseteq\ldots\subseteq \LL=\bigcup_n\LL_n.\end{equation}
We shall see that, up to rescaling the indices, this situation fits in the setting introduced in Section \ref{sect:compositum}.

The following is a corollary of \cite[Proposition 4.12]{ContiTerracini2025}.
Note that the current statement of \emph{loc. cit.} requires $\ovl\rho_{f,p}$ to surject onto $\GL_2(\FF_p)$, but it is clear that the proof only needs that the image of $\ovl\rho_{f,p}$ contains $\SL_2(\FF_p)$.

\begin{proposition}\label{prop:indNC}
In case $\mathrm{(ind)}$, assume that the image of $\ovl\rho_{f,p}$ contains $\SL_2(\FF_p)$. Then, after shifting the indices, the tower \eqref{eq:tower4}
satisfies condition \textup{(NC)} in Definition \ref{ass:tower}.
\end{proposition}

The next proposition summarizes the results of this section.

\begin{proposition}\label{prop:padic}
Let $f\in S_k(\Gamma_0(N))$ be a new Hecke eigenform for some $k\ge 2,N\ge 1$, with rational Hecke field and supercuspidal at $p$. If $\rho_{f,p}\vert_{G_{p}}$ is induced, assume that $p\ge 5$ and that the image of $\ovl\rho_{f,p}$ contains $\SL_2(\FF_p)$. Then, up to a finite rescaling of the indices,   the tower \eqref{eq:tower4} satisfies conditions $\mathrm{(LTR), (NC)}$ of Definition \ref{ass:tower} and condition $\mathrm{(ICE)}$ of Definition \ref{def:ICE}.
\end{proposition}

\begin{proof}
In case (open), $\rho_{f,p}(I_p)$ is open in $\GL_2(\ZZ_p)$. Thus, for all sufficiently large $n$, $\GG_n$ is contained in the image of inertia. Consequently, after shifting the indices, the tower satisfies (LTR) and (NC). Moreover, the proof of \cite[Proposition 3.2]{ContiTerracini2025} gives (ICE), since the central element is obtained from an element of inertia.
 In case (ind), Proposition \ref{prop:indLPTR} implies that the tower satisfies (LTR) after shifting the indices. Under the assumption on the residual image, Proposition \ref{prop:indNC} gives (NC), again after shifting the indices. Choosing $n_0$ sufficiently large and replacing $\LL_n$ by $\LL_{n_0+n}$ for every $n\geq0$, we obtain both conditions simultaneously.

Finally, formula \eqref{eq:suF} and the discussion preceding it show that, if $\rho_{f,p}|_{G_p}$ is induced from a character $\psi$ of $G_{E'}$, then, for a suitable finite extension $F/E'$, we have
\[
(\psi|_{G_F}\circ r_F)(u)=\tau(N_{F/E'}(u))^{1-k}, \qquad (\psi'|_{G_F}\circ r_F)(u)=\tau'(N_{F/E'}(u))^{1-k}
\]
for every $u\in\calo_F^\times$, where $\tau$ and $\tau'$ are the two embeddings of $E'$ into the coefficient field. Since the norm image is open and $k-1\neq 0$, it follows that, in a suitable basis, $\rho_{f,p}(I_p)$ contains
\[
\begin{pmatrix}
\tau(z)&0\\
0&\tau'(z)
\end{pmatrix}
\]
for every $z$ in an open subgroup of $\calo_{E'}^\times$. In particular, it contains an open subgroup of the scalar matrices $\ZZ_p^\times\cdot\mathrm{Id}$.

Let $H\subseteq I_p$ be the subgroup of elements whose image under $\rho_{f,p}$ is scalar. Since $\det\rho_{f,p}=\vareps_p^{k-1}$, where $\vareps_p$ is the $p$-adic cyclotomic character, $\vareps_p(H)$ is an infinite closed subgroup of $\ZZ_p^\times$, hence open. It therefore contains an integer $g>1$. Choosing an element of $H$ with cyclotomic character $g$ proves condition \textup{(ICE)}. See also the last paragraph of the proof of \cite[Proposition 4.12]{ContiTerracini2025}.
\end{proof}

\section{Supercuspidal forms: the adelic property \B}\label{sec:B}

We combine the results of the previous sections to prove the following.

\begin{thm}\label{thm:Bsupercusp}
Let $k\ge 2,N\ge 1$ be two integers, and $f\in S_k(\Gamma_0(N))$ an eigenform such that
\begin{enumerate}
    \item[a)] $\KK_f=\QQ$;
    \item[b)] there exists an odd prime $p$ such that the admissible representation $\pi_{f,p}$ is supercuspidal, and if $\rho_{f,p}\vert_{G_{p}}$ is induced, then $p\ge 5$ and the image of $\ovl\rho_{f,p}$ contains $\SL_2(\FF_p)$.
\end{enumerate}
Then the field $\QQ(\rho_f)$ cut out by the adelic representation
\[\rho_f:G_\QQ \longrightarrow \GL_2(\widehat\ZZ)\]
has property \SB.
\end{thm}

\begin{proof}
Let $\FF$ be the subfield of $\ovl\Q$ fixed by $\ker\rho_f^{(p)}$, $\LL$ the subfield of $\ovl\QQ$ fixed by $\ker\rho_{f,p}$; both $\FF$ and $\LL$ are Galois over $\QQ$, and $\QQ(\rho_f)=\FF\LL$. Consider the tower \eqref{eq:tower4}; after shifting the indices, Proposition \ref{prop:padic} ensures that conditions (LTR), (NC) and (ICE) are satisfied. Observe that $F$ is a discrete valuation field by Proposition \ref{prop:primetop} a); and $\FF/\QQ$ has property \ADZ at $p$ by Corollary \ref{cor:ADZ}, hence the stable Bogomolov property \SB by Proposition \ref{prop:ADZB}. By applying Theorem \ref{cor:composto2} we get the conclusion.\end{proof}

Note that for $k=2$ property \B was already known by \cite{Habegger2013}.

\section{Examples}\label{sec:examples}
To construct explicit examples of modular forms satisfying the hypotheses of Theorem \ref{thm:Bsupercusp}, we have identified two main strategies. 
\subsection{A criterion for ensuring the big residual image condition}
We relied on work of Mascot \cite{Mascot2022} to construct some examples of forms $f$ with rational Hecke field such that the image of $\ovl\rho_{f,p}$ contains $\SL_2(\FF_p)$.
For detecting congruences between forms of possibly different weight, we referred to \cite{Kohnen2004} and \cite[Corollary 9.19]{SteinMod}.

Let $p\geq 5$ be a prime, $N$ a squarefree integer not divisible by $p$, and consider a newform $f\in S_k(\Gamma_0(Np^r))$ with $r\geq 2$, twist minimal (in the sense that $Np^r$ is minimal among the levels of the newforms attached to the twists of $f$ by Dirichlet characters) and with Hecke field $\QQ$. Then $\pi_{f,p}$ is supercuspidal \cite[Proposition 2.8]{LoefflerWeinstein2012}; by Proposition \ref{prop:r=2}, $r=2$. Assume that $\bar\rho_{f,p}$ is irreducible; then it is absolutely irreducible, as explained in \cite[\S 3.3]{Serre1987}. Let $N'$ be the prime-to-$p$ part of the Artin conductor of $\bar\rho_{f,p}$; since $N$ is squarefree, $N'$ is the product of all primes $q\neq p$ such that $\bar\rho_{f,p}$ is ramified at $q$. By the argument in \cite[Lemma 4]{Mascot2022} and the proof of \cite[Theorem 5]{Mascot2022}, if $N'\neq 1$ then the image of $\bar\rho_{f,p}$ contains $\SL_2(\FF_p)$ and Theorem \ref{thm:Bsupercusp} is applicable.
    
By Serre's conjecture, there exists a newform $g$ of level $N'$ such that $\bar\rho_{g,v}\simeq \bar\rho_{f,p}$ for a place $v$ over $p$ in the Hecke field of $g$. The weight of $g$ is an integer $k'\equiv k\mod p-1$  such that $2\leq k'\leq p^2-1$. It follows that if there are no newforms $g\in S_{k'}(\SL_2(\ZZ))$ with $k'\equiv k\mod p-1$, $2\leq k'\leq p^2-1$ such that $\bar\rho_{g,v}\simeq \bar\rho_{f,p}$ for a place $v$ over $p$, then the image of $\bar\rho_{f,p}$ contains $\SL_2(\FF_p)$.

\begin{ex}
Let $p=5$ and consider the twist-minimal newform
$f\in S_4(\Gamma_0(575))$ with rational Hecke field
and label 575.4.a.c in \cite{LMFDB}. We have $575=25\cdot 23$; moreover $a_{17}\equiv 0 \mod 5$ and $X^2+17^3$ does not have roots in $\FF_5$. This shows that $\bar\rho_{f,p}$ is irreducible. A direct search shows that there are no congruences modulo a place above \(5\) of $f$ with forms in $S_{k^\prime}(\SL_2(\ZZ))$, for ${k^\prime}\in\{4,8,12,16,20,24\}$. Then the image of $\bar\rho_{f,p}$ contains $\SL_2(\FF_p)$.
\end{ex}

\subsection{A criterion for excluding local induction}
The big image check is only required when $\rho_{f,p}|_{G_p}$ is induced. Unfortunately, whether this is the case or not cannot be read off directly from the usual classical data attached to $f$. Nevertheless, one can sometimes rule out this possibility on the basis of some indirect considerations.

Let $\pi_{f,p}$ be supercuspidal, with associated admissible pair $(\QQ_{p^2},\theta)$, where  $\theta$ has conductor $1$. Put
\(
m=\bigl|\theta(\ZZ_{p^2}^\times)\bigr|
\); then $m\mid p^2-1$ and the regularity of $\theta$ implies that $m\nmid p-1$.

If in addition $m\nmid 2(p-1)$, then the character $ \theta/ {\theta^{\sigma}}$ does not have order 2 and the projective image of $WD(\pi_{f,p})$ is not $C_2\times C_2$; then the quadratic inducing field of $WD(\pi_{f,p})$ is unique (cf. the classification of the possible projective images in \cite[\S 2.6]{GerardinKutzko1980}).

Suppose that
\[
\rho_{f,p}|_{G_p}\simeq
\mathrm{Ind}_{G_{E'}}^{G_p}\psi
\]
for a quadratic extension $E'/\QQ_p$. Compatibility of the Weil--Deligne functor with induction and local--global compatibility give
\[
WD(\pi_{f,p})
\simeq \mathrm{Ind}_{W_{E'}}^{W_{\QQ_p}}WD(\psi).
\]
Under the above conditions we have $E'=\QQ_{p^2}$. After replacing $\psi$ by its conjugate if necessary, we may assume that $WD(\psi)=\theta$.

It follows that  $\psi$ becomes crystalline over a finite totally ramified extension $F/\QQ_{p^2}$ of degree $m$. As shown in Section \ref{sect:induced}, for a suitable embedding $\tau$ of $\QQ_{p^2}$ into the coefficient field,
\[
(\psi|_{G_F}\circ r_F)(u)=\tau\bigl(N_{F/\QQ_{p^2}}(u)\bigr)^{1-k} \qquad (u\in\calo_F^\times).
\]
The image of $\overline{\psi}|_{I_F}$ is therefore $(\FF_{p^2}^{\times})^{m(k-1)}$, which is not contained in $\FF_p^\times$ if $p+1\nmid m(k-1)$. In this case, $\bar\rho_{f,p}|_{I_p}$ is the sum of two distinct  characters of level $2$ conjugate under Frobenius, so that $\bar\rho_{f,p}|_{G_p}$ is irreducible by \cite[\S 2.2]{Serre1987}. In particular, $f$ cannot be congruent modulo a place $v$ above $p$ to a newform $g$ that is ordinary at $p$.

We deduce the following
\begin{proposition}
Consider a modular newform $f\in S_k(\Gamma_0(N))$ with a prime $p\geq 5$ such that
\begin{enumerate}
\item[a)] the Hecke field of $f$ is $\QQ$ and $\pi_{f,p}$ is supercuspidal;
\item[b)] the admissible pair parameterizing $\pi_{f,p}$ is $(\Q_{p^2},\theta)$ and letting $m=\vert\theta(\Z_{p^2}^\times)\vert$, we have $p+1\nmid m(k-1)$ and $m\nmid 2(p-1)$;
\item[c)] there exist a newform $g$ and a place $v$ above $p$ of its Hecke field such that $g$ is ordinary at $v$ and $\overline{\rho}_{g,v}\simeq\overline{\rho}_{f,p}$.
\end{enumerate}
Then $\rho_{f,p}|_{G_p}$ is not induced from a character of a quadratic extension of $\QQ_p$.
\end{proposition}

\begin{ex}
Let $p=5$ and consider the $p$-supercuspidal newform $f\in S_6(\Gamma_0(75))$ with label 75.6.a.d in \cite{LMFDB}. Its residual representation $\overline{\rho}_{f,5}$ is isomorphic to the residual representation $\overline\rho_{g,5}$ attached to the newform $g\in S_6(\Gamma_0(3))$ with label 3.6.a.a. Since $a_5(g)=6$, the form $g$ is ordinary at $5$; therefore $\overline{\rho}_{f,5}|_{G_{\QQ_5}}$ is reducible.

A Sage computation gives the admissible pair $(E,\theta)$, where $E=\QQ_5(s)$ with $s^2+4s+2=0$, and $\theta$ has conductor $1$ and satisfies
\[
\theta(s)=-d-1, \hbox{ where }
 d^2+d+1=0.
\]
The extension $E/\QQ_5$ is unramified, and the reduction of $s$ generates $\FF_{25}^{\times}$. Since $\theta(s)$ has order $3$, we obtain
\[
m=|\theta(\calo_E^\times)|=3.
\]
Thus $p+1=6$ does not divide $m(k-1)=15$, and $m=3$ does not divide $2(p-1)=8$. The preceding proposition implies that $\rho_{f,5}|_{G_{\QQ_5}}$ is not induced. We can therefore apply Theorem \ref{thm:Bsupercusp}.
\end{ex}

\bibliographystyle{abbrv}
\bibliography{BOG}

@article{AmorosoTerracini2024,
title = {Bogomolov property and {G}alois representations},
journal = {J. Number Theory},
volume = {279},
pages = {294-322},
year = {2026},
issn = {0022-314X},
doi = {https://doi.org/10.1016/j.jnt.2025.06.002},
url = {https://www.sciencedirect.com/science/article/pii/S0022314X25001891},
author = {Francesco Amoroso and Lea Terracini},
}

@article {AmorosoDavidZannier2014,
    AUTHOR = {Amoroso, Francesco and David, Sinnou and Zannier, Umberto},
     TITLE = {On fields with {P}roperty ({B})},
   JOURNAL = {Proc. Amer. Math. Soc.},
  FJOURNAL = {Proceedings of the American Mathematical Society},
    VOLUME = {142},
      YEAR = {2014},
    NUMBER = {6},
     PAGES = {1893--1910},
      }

@article {AmorosoDvornicich2000,
    AUTHOR = {Amoroso, Francesco and Dvornicich, Roberto},
     TITLE = {A lower bound for the height in abelian extensions},
   JOURNAL = {J. Number Theory},
  FJOURNAL = {Journal of Number Theory},
    VOLUME = {80},
      YEAR = {2000},
    NUMBER = {2},
     PAGES = {260--272},  
     }

@article {AmorosoZannier2000,
    AUTHOR = {Amoroso, Francesco and Zannier, Umberto},
     TITLE = {A relative {D}obrowolski lower bound over abelian extensions},
   JOURNAL = {Ann. Scuola Norm. Sup. Pisa Cl. Sci. (4)},
    VOLUME = {29},
      YEAR = {2000},
    NUMBER = {3},
     PAGES = {711--727},
}

@article {GhateMezard2009,
    AUTHOR = {Ghate, Eknath and M\'ezard, Ariane},
     TITLE = {Filtered modules with coefficients},
   JOURNAL = {Trans. Amer. Math. Soc.},
  FJOURNAL = {Transactions of the American Mathematical Society},
    VOLUME = {361},
      YEAR = {2009},
    NUMBER = {5},
     PAGES = {2243--2261},}

@article {BombieriZannier2001,
    AUTHOR = {Bombieri, Enrico and Zannier, Umberto},
     TITLE = {A note on heights in certain infinite extensions of {$\mathbb
              Q$}},
   JOURNAL = {Atti Accad. Naz. Lincei Cl. Sci. Fis. Mat. Natur. Rend. Lincei
              (9) Mat. Appl.},
  FJOURNAL = {Atti della Accademia Nazionale dei Lincei. Classe di Scienze
              Fisiche, Matematiche e Naturali. Rendiconti Lincei. Serie IX.
              Matematica e Applicazioni},
    VOLUME = {12},
      YEAR = {2001},
     PAGES = {5--14 (2002)}
}

@incollection {Blasius2006,
    AUTHOR = {Blasius, Don},
     TITLE = {Hilbert modular forms and the {R}amanujan conjecture},
 BOOKTITLE = {Noncommutative geometry and number theory},
    SERIES = {Aspects Math.},
    VOLUME = {E37},
     PAGES = {35--56},
 PUBLISHER = {Friedr. Vieweg, Wiesbaden},
      YEAR = {2006},
      ISBN = {3-8348-0170-4},}

@article {CalegariSardari,
    AUTHOR = {Calegari, Frank and Talebizadeh Sardari, Naser},
     TITLE = {Vanishing {F}ourier coefficients of {H}ecke eigenforms},
   JOURNAL = {Math. Ann.},
  FJOURNAL = {Mathematische Annalen},
    VOLUME = {381},
      YEAR = {2021},
    NUMBER = {3-4},
     PAGES = {1197--1215}}

@article{Carayol1986,
 author = {Carayol, Henri},
 title = {Sur les repr{é}sentations {$\ell$}-adiques associ{é}es aux formes
modulaires de {Hilbert}},
 fjournal = {Annales Scientifiques de l'{\'E}cole Normale Sup{\'e}rieure. Quatri{\`e}me S{\'e}rie},
 journal = {Ann. Sci. {\'E}c. Norm. Sup{\'e}r. (4)},
 issn = {0012-9593},
 volume = {19},
 number = {3},
 pages = {409--468},
 year = {1986},
}

@misc {ConradLifting2012,
	AUTHOR = {Conrad, Brian},
	TITLE = {Lifting global representations with local properties},
	YEAR = {2012},
howpublished={http://math.stanford.edu/~conrad/papers/locchar.pdf}}

@misc{ContiDelCorsoPlessisTerracini2026,
      title={Small points in radical extensions of number fields}, 
      author={Andrea Conti and Ilaria Del Corso and Arnaud Plessis and Lea Terracini},
      year={2026},
      eprint={2607.29208},
      archivePrefix={arXiv},
      primaryClass={math.NT},
      howpublished={https://arxiv.org/abs/2607.29208}, 
}

@misc{ContiTerracini2025,
      title={Bogomolov property for {G}alois representations with big local image}, 
      author={Andrea Conti and Lea Terracini},
      year={2025},
      eprint={2503.14052},
      archivePrefix={arXiv},
      primaryClass={math.NT},
      url={https://arxiv.org/abs/2503.14052}, 
howpublished={https://arxiv.org/abs/2403.03319},
}

@misc{Piras2026,
      title={Bogomolov property for modular {G}alois representations with nontrivial nebentypus}, 
      author={Pietro Piras},
      year={2026},
      eprint={2603.13523},
      archivePrefix={arXiv},
      primaryClass={math.NT},
      howpublished={https://arxiv.org/abs/2603.13523}, 
}

@incollection {FontaineMazur1995,
    AUTHOR = {Fontaine, Jean-Marc and Mazur, Barry},
     TITLE = {Geometric {G}alois representations},
 BOOKTITLE = {Elliptic curves, modular forms, \& {F}ermat's last theorem
              ({H}ong {K}ong, 1993)},
    SERIES = {Ser. Number Theory},
    VOLUME = {I},
     PAGES = {41--78},
 PUBLISHER = {Int. Press, Cambridge, MA},
      YEAR = {1995},
}

@article {Frey2021,
    AUTHOR = {Frey, Linda},
     TITLE = {Explicit small heights in infinite non-abelian extensions},
   JOURNAL = {Acta Arith.},
    VOLUME = {199},
      YEAR = {2021},
    NUMBER = {2},
     PAGES = {111--133},
}

@incollection {Gerardin1978,
    AUTHOR = {G\'erardin, Paul},
     TITLE = {Facteurs locaux des alg\`ebres simples de rang {$4$}. {I}},
 BOOKTITLE = {Reductive groups and automorphic forms, {I} ({P}aris,
              1976/1977)},
    SERIES = {Publ. Math. Univ. Paris VII},
    VOLUME = {1},
     PAGES = {37--77},
 PUBLISHER = {Univ. Paris VII, Paris},
      YEAR = {1978},
   MRCLASS = {22E50 (12B27)},
  MRNUMBER = {680785},
MRREVIEWER = {Joe\ Repka},
}

@article{GerardinKutzko1980,
author = {G\'erardin, Paul and Kutzko, Philip},
journal = {Annales scientifiques de l'École Normale Supérieure},
number = {3},
pages = {349-384},
title = {Facteurs locaux pour $\mathrm {GL}(2)$},
volume = {13},
year = {1980},
}

@article {Habegger2013,
    AUTHOR = {Habegger, P.},
     TITLE = {Small height and infinite nonabelian extensions},
   JOURNAL = {Duke Math. J.},
  FJOURNAL = {Duke Mathematical Journal},
    VOLUME = {162},
      YEAR = {2013},
    NUMBER = {11},
     PAGES = {2027--2076},
}

@article {GhateVatsal2004,
    AUTHOR = {Ghate, Eknath and Vatsal, Vinayak},
     TITLE = {On the local behaviour of ordinary {$\Lambda$}-adic
              representations},
   JOURNAL = {Ann. Inst. Fourier (Grenoble)},
  FJOURNAL = {Universit\'e{} de Grenoble. Annales de l'Institut Fourier},
    VOLUME = {54},
      YEAR = {2004},
    NUMBER = {7},
     PAGES = {2143--2162}}

@article {Kohnen2004,
    AUTHOR = {Kohnen, Winfried},
     TITLE = {On {F}ourier coefficients of modular forms of different
              weights},
   JOURNAL = {Acta Arith.},
  FJOURNAL = {Acta Arithmetica},
    VOLUME = {113},
      YEAR = {2004},
    NUMBER = {1},
     PAGES = {57--67},}

@misc{LMFDB,
  shorthand    = {LMFDB},
  author       = { {The LMFDB Collaboration}},
  title        = {The {L}-functions and modular forms database},
  howpublished = {\url{https://www.lmfdb.org}},
  year         = {2024},
}

@article{LoefflerWeinstein2012,
 author = {Loeffler, David and Weinstein, Jared},
 title = {On the computation of local components of a newform},
 fjournal = {Mathematics of Computation},
 journal = {Math. Comput.},
 issn = {0025-5718},
 volume = {81},
 number = {278},
 pages = {1179--1200},
 year = {2012},
}

@article {LoefflerAdelic2017,
    AUTHOR = {Loeffler, David},
     TITLE = {Images of adelic {G}alois representations for modular forms},
   JOURNAL = {Glasg. Math. J.},
  FJOURNAL = {Glasgow Mathematical Journal},
    VOLUME = {59},
      YEAR = {2017},
    NUMBER = {1},
     PAGES = {11--25}}

@article {MandalMondal2026,
    AUTHOR = {Mandal, Tathagata and Mondal, Sudipa},
     TITLE = {On the change of epsilon factors for symmetric square
              transfers under twisting and applications},
   JOURNAL = {Monatsh. Math.},
  FJOURNAL = {Monatshefte f\"ur Mathematik},
    VOLUME = {209},
      YEAR = {2026},
    NUMBER = {1},
     PAGES = {105--132},}

@misc{Mascot2022,
      title={A method to prove that a modular {G}alois representation has large image}, 
      author={Nicolas Mascot},
      year={2022},
      eprint={2205.14030},
      archivePrefix={arXiv},
      primaryClass={math.NT},
      howpublished={https://arxiv.org/abs/2205.14030}, 
}

@techreport{Probst2008,
  author      = {Probst, C.},
  title       = {Filtrations of {$p$}-adic analytic {G}alois Groups of Local Fields},
type={PhD
thesis},
  institution = {Universit\"at Heidelberg},
  year        = {2008}
}

@article {Saito1997,
    AUTHOR = {Saito, Takeshi},
     TITLE = {Modular forms and {$p$}-adic {H}odge theory},
   JOURNAL = {Invent. Math.},
  FJOURNAL = {Inventiones Mathematicae},
    VOLUME = {129},
      YEAR = {1997},
    NUMBER = {3},
     PAGES = {607--620},}

@article {Sen1972,
    AUTHOR = {Sen, Shankar},
     TITLE = {Ramification in {$p$}-adic {L}ie extensions},
   JOURNAL = {Invent. Math.},
  FJOURNAL = {Inventiones Mathematicae},
    VOLUME = {17},
      YEAR = {1972},
     PAGES = {44--50},
}

@article {Sen1973,
    AUTHOR = {Sen, Shankar},
     TITLE = {Lie algebras of {G}alois groups arising from {H}odge-{T}ate
              modules},
   JOURNAL = {Ann. of Math. (2)},
  FJOURNAL = {Annals of Mathematics. Second Series},
    VOLUME = {97},
      YEAR = {1973},
     PAGES = {160--170},}

@article {SerreDiv1976,
    AUTHOR = {Serre, Jean-Pierre},
     TITLE = {Divisibilit\'{e} de certaines fonctions arithm\'{e}tiques},
   JOURNAL = {Enseign. Math. (2)},
  FJOURNAL = {L'Enseignement Math\'{e}matique. Revue Internationale. 2e
              S\'{e}rie},
    VOLUME = {22},
      YEAR = {1976},
    NUMBER = {3-4},
     PAGES = {227--260},}

@book {SerreLie2006,
    AUTHOR = {Serre, Jean-Pierre},
     TITLE = {Lie algebras and {L}ie groups},
    SERIES = {Lecture Notes in Mathematics},
    VOLUME = {1500},
      NOTE = {1964 lectures given at Harvard University,
              Corrected fifth printing of the second (1992) edition},
 PUBLISHER = {Springer-Verlag, Berlin},
      YEAR = {2006},
     PAGES = {viii+168}}

@book {Serre1968,
    AUTHOR = {Serre, Jean-Pierre},
     TITLE = {Corps locaux},
    SERIES = {Publications de l'Universit\'{e} de Nancago},
    VOLUME = {No. VIII},
      NOTE = {Deuxi\`eme \'{e}dition},
 PUBLISHER = {Hermann, Paris},
      YEAR = {1968},
}

@article {Serre1987,
    AUTHOR = {Serre, Jean-Pierre},
     TITLE = {Sur les repr\'{e}sentations modulaires de degr\'{e} {$2$} de
              {$\mathrm{Gal}(\overline{\mathbf{Q}}/\mathbf{Q})$}},
   JOURNAL = {Duke Math. J.},
  FJOURNAL = {Duke Mathematical Journal},
    VOLUME = {54},
      YEAR = {1987},
    NUMBER = {1},
     PAGES = {179--230},
      }

@incollection{Tate1979,
  author    = {Tate, John},
  title     = {Number theoretic background},
  booktitle = {Automorphic Forms, Representations and $L$-Functions},
  editor    = {Borel, Armand and Casselman, William},
  series    = {Proceedings of Symposia in Pure Mathematics},
  volume    = {33-2},
  pages     = {3--26},
  publisher = {American Mathematical Society},
  address   = {Providence, RI},
  year      = {1979},
  note      = {Proc. Sympos. Pure Math., Corvallis, 1977}
}

@book {SteinMod,
    AUTHOR = {Stein, William},
     TITLE = {Modular forms, a computational approach},
    SERIES = {Graduate Studies in Mathematics},
    VOLUME = {79},
      NOTE = {With an appendix by Paul E. Gunnells},
 PUBLISHER = {American Mathematical Society, Providence, RI},
      YEAR = {2007},
     PAGES = {xvi+268},}

@incollection {S08,
    AUTHOR = {Smyth, Chris},
     TITLE = {The {M}ahler measure of algebraic numbers: a survey},
 BOOKTITLE = {Number theory and polynomials},
    SERIES = {London Math. Soc. Lecture Note Ser.},
    VOLUME = {352},
     PAGES = {322--349},
 PUBLISHER = {Cambridge Univ. Press, Cambridge},
      YEAR = {2008},
      ISBN = {978-0-521-71467-9},
   MRCLASS = {11R06},
  MRNUMBER = {2428530},
       DOI = {10.1017/CBO9780511721274.021},
       URL = {https://doi.org/10.1017/CBO9780511721274.021},
}

@incollection {Rem17,
    AUTHOR = {R\'emond, Ga\"el},
     TITLE = {G\'en\'eralisations du probl\`eme de {L}ehmer et applications
              \`a{} la conjecture de {Z}ilber-{P}ink},
 BOOKTITLE = {Around the {Z}ilber-{P}ink conjecture/{A}utour de la
              conjecture de {Z}ilber-{P}ink},
    SERIES = {Panor. Synth\`eses},
    VOLUME = {52},
     PAGES = {243--284},
 PUBLISHER = {Soc. Math. France, Paris},
      YEAR = {2017},
      ISBN = {978-2-85629-856-5},
   MRCLASS = {11G50 (11G10 14K15)},
  MRNUMBER = {3728315},
MRREVIEWER = {Ariyan\ Javanpeykar},
}

@article {Amo16,
    AUTHOR = {Amoroso, Francesco},
     TITLE = {On a conjecture of {G}. {R}\'emond},
   JOURNAL = {Ann. Sc. Norm. Super. Pisa Cl. Sci. (5)},
  FJOURNAL = {Annali della Scuola Normale Superiore di Pisa. Classe di
              Scienze. Serie V},
    VOLUME = {15},
      YEAR = {2016},
     PAGES = {599--608},
      ISSN = {0391-173X,2036-2145},
   MRCLASS = {11G50 (12E30)},
  MRNUMBER = {3495440},
MRREVIEWER = {Dragos\ Ghioca},
}

@misc{Gri17,
 author = {Grizzard, Robert},
 title = {Remarks on {R{\'e}mond}'s generalized {Lehmer} problems},
 year = {2017},
 howpublished = {Preprint, {arXiv}:1710.11614 [math.{NT}] (2017)},
 url = {https://arxiv.org/abs/1710.11614},
 arXiv = {arXiv:1710.11614}
}

@article {Pot21,
    AUTHOR = {Pottmeyer, Lukas},
     TITLE = {Fields generated by finite rank subgroups of {$\overline{\mathbb
              {Q}}^*$}},
   JOURNAL = {Int. J. Number Theory},
  FJOURNAL = {International Journal of Number Theory},
    VOLUME = {17},
      YEAR = {2021},
    NUMBER = {5},
     PAGES = {1079--1089},
      ISSN = {1793-0421,1793-7310},
   MRCLASS = {11G50 (11R04 14K05 20K15)},
  MRNUMBER = {4270874},
MRREVIEWER = {Francesco\ Campagna},
       DOI = {10.1142/S1793042121500251},
       URL = {https://doi.org/10.1142/S1793042121500251},
}

@article {Ple24,
    AUTHOR = {Plessis, Arnaud},
     TITLE = {Location of small points on an elliptic curve by an
              equidistribution argument},
   JOURNAL = {Int. Math. Res. Not. IMRN},
  FJOURNAL = {International Mathematics Research Notices. IMRN},
      YEAR = {2024},
      Volume={},
    NUMBER = {6},
     PAGES = {4689--4709},
      ISSN = {1073-7928,1687-0247},
   MRCLASS = {11G05 (11K06)},
  MRNUMBER = {4721653},
MRREVIEWER = {M.\ Ram\ Murty},
       DOI = {10.1093/imrn/rnad051},
       URL = {https://doi.org/10.1093/imrn/rnad051},
}

@article {Ple22,
    AUTHOR = {Plessis, Arnaud},
     TITLE = {Points de petite hauteur sur une vari\'et\'e{}
              semi-ab\'elienne de la forme {$\mathbb {G}^n_m\times A$}},
   JOURNAL = {Bull. Lond. Math. Soc.},
  FJOURNAL = {Bulletin of the London Mathematical Society},
    VOLUME = {54},
      YEAR = {2022},
    NUMBER = {6},
     PAGES = {2278--2296},
      ISSN = {0024-6093,1469-2120},
   MRCLASS = {11G50 (11G10 14G40)},
  MRNUMBER = {4528623},
MRREVIEWER = {Chunhui\ Liu},
       DOI = {10.1112/blms.12693},
       URL = {https://doi.org/10.1112/blms.12693},
}

@article{Ple24a,
 author = {Plessis, Arnaud},
 title = {A new way to tackle a conjecture of {R{\'e}mond}},
 fjournal = {Canadian Journal of Mathematics},
 journal = {Can. J. Math.},
 issn = {0008-414X},
 volume = {76},
 number = {6},
 pages = {2049--2072},
 year = {2024},
 language = {English},
 doi = {10.4153/S0008414X2300072X},
 zbMATH = {7955688},
 Zbl = {1567.11082}
}

@article {S10,
    AUTHOR = {Silverman, Joseph H.},
     TITLE = {Lang's height conjecture and {S}zpiro's conjecture},
   JOURNAL = {New York J. Math.},
  FJOURNAL = {New York Journal of Mathematics},
    VOLUME = {16},
      YEAR = {2010},
     PAGES = {1--12},
      ISSN = {1076-9803},
   MRCLASS = {11G50 (11G05 11J97)},
  MRNUMBER = {2645981},
MRREVIEWER = {Machiel\ van Frankenhuijsen},
       URL = {http://nyjm.albany.edu:8000/j/2010/16_1.html},
}

@article {Pl24,
    AUTHOR = {Plessis, Arnaud},
     TITLE = {Bogomolov property of some infinite nonabelian extensions of a
              totally {$v$}-adic field},
   JOURNAL = {Acta Arith.},
  FJOURNAL = {Acta Arithmetica},
    VOLUME = {213},
      YEAR = {2024},
    NUMBER = {1},
     PAGES = {1--23},
      ISSN = {0065-1036,1730-6264},
   MRCLASS = {11G50 (11G05)},
  MRNUMBER = {4742648},
       DOI = {10.4064/aa220216-31-1},
       URL = {https://doi.org/10.4064/aa220216-31-1},
}

\smallskip

\begin{footnotesize}
\noindent\textsc{Andrea Conti} -- IWR, Heidelberg University, Im Neuenheimer Feld 205, 69120 Heidelberg, Germany, \url{andrea.conti@iwr.uni-heidelberg.de} \\
\noindent\textsc{Pietro Piras} -- Dipartimento di Matematica, Università di Torino, Via Carlo Alberto, 10 - 10123 Torino, Italy, \url{pietro.piras@unito.it} \\
\noindent\textsc{Lea Terracini} -- Dipartimento di Informatica, Università di Torino, Corso Svizzera 185, 10149 Torino, Italy, \url{lea.terracini@unito.it}
\end{footnotesize}

\end{document}